\documentclass{article}
\usepackage[T1]{fontenc}
\usepackage[margin=1in]{geometry}
\usepackage{graphicx}
\usepackage{amsmath, amsfonts, amssymb, amsthm, stmaryrd, bm, bbm, doi, authblk}
\usepackage{pgfplots}
\usepgfplotslibrary{colormaps}
\usepackage{hyperref,cleveref,csquotes,color,enumitem}
\hypersetup{colorlinks=true,breaklinks=true,urlcolor=black,linkcolor=black,citecolor=black}

\newtheorem{theorem}{Theorem}[section]
\newtheorem{lemma}[theorem]{Lemma}
\crefname{lemma}{Lemma}{Lemmas}
\newtheorem{proposition}[theorem]{Proposition}
\crefname{proposition}{Proposition}{Propositions}
\newtheorem{corollary}[theorem]{Corollary}
\crefname{corollary}{Corollary}{Corollaries}
\theoremstyle{remark}
\newtheorem{example}[theorem]{Example}
\crefname{example}{Example}{Examples}
\newtheorem{remark}[theorem]{Remark}
\crefname{remark}{Remark}{Remarks}
\AtBeginEnvironment{remark}{
\pushQED{\qed}%
}
\AtEndEnvironment{remark}{\popQED}

\newcommand{\bN}{\mathbb{N}}
\newcommand{\bZ}{\mathbb{Z}}

\newcommand{\bR}{\mathbb{R}}
\newcommand{\abs}[1]{\left|#1\right|}
\newcommand{\card}[1]{\#\left\{#1\right\}}

\newcommand{\cv}[1]{\underset{#1}{\;\longrightarrow}\;}

\newcommand{\cO}{\mathcal{O}}
\newcommand{\cC}{\mathcal{C}}

\newcommand{\cI}{\mathcal{I}}
\newcommand{\unmatched}{{\textrm{unmatched}}}
\newcommand{\bP}{\mathbb{P}}
\newcommand{\prob}[1]{\bP{\left(#1\right)}}
\newcommand{\probcond}[2]{\bP{\left(\left. #1 \,\right|\, #2 \right)}}

\newcommand{\expec}[1]{\mathbb{E}{\left[#1\right]}}
\newcommand{\expecond}[2]{\mathbb{E}{\left[\left. #1 \,\right|\, #2 \right]}}

\newcommand{\Leb}{\mathrm{Leb}}
\newcommand{\One}[1]{\mathbbm{1}_{#1}}
\newcommand{\Unif}[1]{\mathrm{Unif}{\left(#1\right)}}

\newcommand{\Exp}[1]{{\rm Exp}{\left(#1\right)}}
\newcommand{\Binomial}[2]{\mathrm{Bin}{\left(#1,#2\right)}}
\newcommand{\fS}{\mathfrak{S}}
\DeclareMathOperator{\id}{id}
\newcommand{\LIS}[1]{\mathrm{LIS}{\left(#1\right)}}
\newcommand{\LDS}[1]{\mathrm{LDS}{\left(#1\right)}}
\newcommand{\Luce}[1]{{\rm Luce}{\left(#1\right)}}
\newcommand{\bth}{\bm{\theta}}
\newcommand{\Perm}[1]{\mathrm{Perm}{\left(#1\right)}}
\def\sample{\mathrm{Sample}}
\newcommand{\bU}{{\mathbf{U}}}
\newcommand{\cP}{\mathcal{P}}

\title{\textsc{The longest increasing subsequences of Luce permutations and permutons}}

\author[$*$]{Jacopo Borga}
\affil[$*$]{Department of Mathematics, Massachusetts Institute of Technology}

\author[$\dagger$]{Victor Dubach}
\affil[$\dagger$]{Institut Camille Jordan, Université Lyon 1}

\date{}

\begin{document}

\maketitle

\begin{abstract}
    We investigate the longest increasing subsequences (LIS) of random permutations distributed under the Luce model with weights $\theta_1^{(n)} \ge \ldots \ge \theta_n^{(n)}$.
    Our first main result provides a general formula for the growth rate of the LIS.
    For example, we get that the LIS of Luce permutations with Sukhatme weights $\theta_k^{(n)} = n-k+1$ grows like $\sqrt{n}$.
    We also give simple conditions for the LIS to be of order $n^\beta$, for any given $\beta\in[1/2,1)$.
    Our second main result states that, if the weights are of the form $\theta_k^{(n)} = f\big( (k-1)/n \big)$ for a sufficiently regular function $f$, then the LIS divided by $2\sqrt{n}$ converges in probability to a constant $K(f)$ described by a variational problem.
    Throughout, we make extensive use of coupling techniques;
    in particular, to prove our second main result, we argue that Luce random permutations behave similarly to random permutations sampled from the ``Luce permuton''.
\end{abstract}

\tableofcontents

\section{Introduction}

\subsection{Longest increasing subsequences of permutations}

For $n\in\bN$, $\sigma\in\fS_n$ a permutation of size $n$, and $k\le n$, we call an {increasing subsequence} of length $k$ in $\sigma$ a sequence $\{ (i_1, \sigma(i_1)) , \ldots , (i_k, \sigma(i_k)) \}$ such that $i_1 < \ldots < i_k$ and $\sigma(i_1) < \ldots < \sigma(i_k)$.
Then, the length of the \emph{longest increasing subsequence} in $\sigma$, or simply LIS of $\sigma$, is the largest integer $k$ such that $\sigma$ contains an increasing subsequence of length $k$.
We let $\LIS{\sigma}$ denote this integer.

The LIS, for instance, provides a natural way to compare rankings through the \emph{Ulam distance}; see \cite[Chapter~6, Section~B, pp.~118--119]{Diaconis_1988} or \cite{Aldous_Diaconis_1999}.
Viewing permutations $\sigma,\tau\in\fS_n$ as maps assigning ranks to items, this distance is given by
\[
    d_{\mathrm U}(\sigma,\tau) := n-\LIS{\sigma\circ\tau^{-1}}.
\]
It is the minimum number of moves needed to transform one ranked list into the other, where a move removes a single item and reinserts it elsewhere.
In particular, $d_{\mathrm U}(\sigma,\mathrm{id})=n-\LIS{\sigma}$, so studying the LIS of a random ranking gives direct information about its distance from the increasing order, a natural measure of disarray.

The study of longest increasing subsequences in random permutations was initiated by Ulam, who asked in \cite{Ulam_61} about the asymptotic behavior of $\ell_n := \expec{\LIS{\sigma_n}}$ when $\sigma_n$ is uniformly distributed over $\fS_n$.
Thanks to major developments \cite{Logan_Shepp_1977,Vershik_Kerov_1977,Aldous_Diaconis_1995}, this question has been solved at the first order:

\begin{theorem}\label{th: LIS uniform 2sqrt n}
    For each $n\in\bN$, let $\sigma_n \sim \Unif{\fS_n}$. 
    Then,
    $
        \frac{1}{\sqrt n} \LIS{\sigma_n} \longrightarrow 2
    $
    as $n\to\infty$
    in probability and with moments.
    In particular, writing $\ell_n := \expec{\LIS{\sigma_n}}$, we get that $\ell_n / \sqrt n \longrightarrow 2$ as $n\to\infty$.
\end{theorem}

The next major breakthrough was brought by Baik--Deift--Johansson, who proved in \cite{Baik_Deift_Johansson_1999} that the fluctuations of $\LIS{\sigma_n}$ were described by the Tracy--Widom distribution.
Recently, Dauvergne and Vir{\'a}g \cite{Dauvergne_Virag_2022} related the \emph{shape} of the longest increasing subsequence to the so-called directed geodesic of the directed landscape.

In recent years, a lot of work has been done to analyze the LIS of random permutations with \emph{non-uniform} distributions.
These encompass, e.g., the Mallows distribution \cite{Mueller_Starr_2013,Bhatnagar_Peled_2015,Basu_Bhatnagar_2017}, the Ewens distribution and other conjugacy-invariant models \cite{Kammoun_2018,Dubach_2025}, permutations sampled from a density on the unit square \cite{Deuschel_Zeitouni_1995,Dubach_2023}, Baxter permutations~\cite{bgsp-26-Tightness}, or permutations sampled from the Brownian separable permutons \cite{Borga_DaSilva_Gwynne_2024,Adhikari_Borga_Budzinski_Silva_Senizergues_2025}.
In this paper, we contribute to this line of research by considering the \emph{Luce distribution}.

\subsection{The Luce distribution}

The Luce distribution describes a basic sampling procedure: draw balls from a weighted urn, one at a time and without replacement, choosing each remaining ball with probability proportional to its weight.
The order of the draws produces a random ranking.

More precisely, let $\bth = (\theta_1, \dots, \theta_n)$ be a sequence of positive weights with total sum $w_n := \theta_1+ \ldots+ \theta_n$.
We say that a random permutation $\tau_n$ is Luce distributed with weights $\bth$, denoted $\tau_n \sim \Luce{\bth}$, if:
\[
\text{for each } \sigma\in\fS_n,\quad
\prob{\tau_n = \sigma} = 
\frac{\theta_{\sigma^{-1}(1)}}{w_n} \times 
\frac{\theta_{\sigma^{-1}(2)}}{w_n - \theta_{\sigma^{-1}(1)}} \times 
\frac{\theta_{\sigma^{-1}(3)}}{w_n - \theta_{\sigma^{-1}(1)} - \theta_{\sigma^{-1}(2)}} \times 
\cdots \times
\frac{\theta_{\sigma^{-1}(n)}}{\theta_{\sigma^{-1}(n)}} \,.
\]
In terms of the urn construction, ball $i$ has weight $\theta_i$, and $\tau_n(i)$ records the draw at which it is selected.
Thus, $\tau_n^{-1}(1),\ldots,\tau_n^{-1}(n)$ are the labels of the balls in their order of selection.
Note that rescaling the weights does not affect the resulting distribution, i.e., $\Luce{\theta_1, \ldots, \theta_n} = \Luce{\alpha\theta_1, \ldots, \alpha\theta_n}$ for any $\alpha>0$, and thus, we can freely rescale the weights as convenient.
Some examples are given by:
\begin{itemize}
    \item The ``Sukhatme'' weights $\theta_k = n-k+1$;
    \item The ``polynomial Sukhatme'' weights $\theta_k = (n-k+1)^\alpha$ for a fixed $\alpha>0$;
    \item The ``exponential Sukhatme'' weights $\theta_k = \alpha^{k}$ for a fixed $\alpha>0$;
    \item The ``expoly-Sukhatme'' weights $\theta_k = \exp{\left[ - {\left( \frac{n-k+1}{n} \right)}^{-\alpha} \right]}$ for a fixed $\alpha>0$; 
    \item The ``Mallows-like'' weights $\theta_k = q_n^k$, for a given $q_n$ (e.g.\ $q_n = e^{-\beta/n}$ for a fixed $\beta\in\bR$, see \Cref{sec: discussion}).
\end{itemize}
An important case that will accompany us throughout the paper is when $\theta_k = f((k-1)/n)$ for some positive function $f$ on $[0,1)$.
This is e.g.\ the case for the Sukhatme, polynomial Sukhatme and expoly-Sukhatme weights above, and also the Mallows-like weights in the ``thermodynamic regime'' $q_n=e^{-\beta/n}$.

The above distribution was introduced by Luce (see, e.g., \cite{Luce_1977}) to model rankings of data given to subjects in psychophysics experiments.
This elementary sampling scheme, known as \emph{successive sampling}, is a standard model in survey sampling~\cite{Gordon_1983}.
It also underlies models used in horse-race betting to predict finishing orders from winning probabilities~\cite{Harville_1973}.
Another important occurrence is the stationary distribution of the \emph{Tsetlin library}: at each step, book $i$ is chosen with probability $\theta_i/w_n$ and moved to the front of the shelf; at stationarity, the shelf order has the law of $\tau_n^{-1}$~\cite[Section~2.2.3]{chatterjee2024enumerative}.

As basic as this sampling mechanism is, many natural features of the resulting rankings remain to be understood.
We hope that our approach to studying the LIS, and hence Ulam distance, will also help in the study of other statistics of applied interest.
Recently, Chatterjee--Diaconis--Kim analyzed in \cite{chatterjee2024enumerative} the distributions of $(\tau_n^{-1}(1), \ldots, \tau_n^{-1}(k))$ (the labels on the first $k$ balls) and of $(\tau_n^{-1}(n), \ldots, \tau_n^{-1}(n-k+1))$ (the labels on the last $k$ balls) for $k\in\bN$.
They brought attention to a range of questions which remained open, such as:
    ``What is the asymptotic behavior of the length of the longest increasing subsequence of $\tau_n$?''.
Later, the first author here together with Chatterjee and Diaconis computed in \cite{Borga_Chatterjee_Diaconis_2025} the \emph{permuton limit} of $\tau_n$ (see the next section).
In addition, they computed the local limit of the model and provided CLTs for consecutive pattern counts and for the number of inversions.
Our goal here is to settle the question of the LIS of $\tau_n$.

\begin{remark}
    Regarding the definition of the Luce distribution, we adopt here the same convention as in \cite{Borga_Chatterjee_Diaconis_2025}, which is the inverse of the convention in \cite{chatterjee2024enumerative}.
    Nevertheless, this choice does not affect our results since the LIS is invariant under the inverse map.
\end{remark}

Throughout our proofs, we will heavily rely on an alternative construction of $\tau_n$ using exponential random variables.
If $(y_1, \dots, y_n)$ is a sequence of pairwise distinct real numbers, it shares its relative ordering with a unique permutation $\sigma = (\sigma_1, \dots, \sigma_n)$ in $\fS_n$, that is,
$
\sigma_i = \sum_{1\le j\le n} \One{y_j \le y_i} 
$
for each $i\in\{1,\ldots,n\}$.
We let $\Perm{y_1, \ldots, y_n}$ denote this permutation, and call it the \emph{permutation induced by $(y_1, \dots, y_n)$}.
The next lemma can be found in \cite[Theorem~2.1]{chatterjee2024enumerative}.

\begin{lemma}[Exponential representation]\label{lem: exponential representation}
    Let $(\theta_i)_{1\le i\le n}$ be a sequence of positive numbers.
    Let $(E_i)_{1\le i\le n}$ be a sequence of independent exponential random variables with parameters $(\theta_i)_{1\le i\le n}$ (i.e., with means $(\theta_i^{-1})_{1\le i\le n}$).
    Then, the random permutation  $\Perm{E_1, \dots, E_n}$ follows the $\Luce{\theta_1, \dots, \theta_n}$ distribution.
\end{lemma}

\subsection{Permuton limits and samples}

Permutons were formally introduced in \cite{Hoppen_Kohayakawa_Moreira_Rath_MenezesSampaio_2013} as scaling limits for sequences of permutations.
A probability measure $\mu$ on the unit square $[0,1]^2$ is called a permuton if both its marginals are uniform, that is, 
$
\mu([0,t]\times[0,1]) = \mu([0,1]\times[0,t]) = t $
for all $t\in[0,1]$.
Any permutation $\sigma \in \fS_n$ can be mapped to a permuton $\mu_\sigma$ with density
\[
\frac{\rm d\mu_\sigma}{\rm d\Leb} := n \sum_{1\le i\le n} \One{[\frac{i-1}{n} , \frac{i}{n}] \times [\frac{\sigma(i)-1}{n} , \frac{\sigma(i)}{n}]} \,,
\]
and we say that a sequence of permutations $(\sigma_n)_{n\in\bN}$ converges to a permuton $\mu$ if $(\mu_{\sigma_n})_{n\in\bN}$ converges to $\mu$ weakly.
The theory of permutons has been studied quite intensively over the past decade; see, e.g., \cite{bfw-permutons-survey} for an overview.

Any given permuton $\mu$ gives rise to an important sequence of random permutations,
sometimes called \emph{$\mu$-random permutations} or \emph{permuton samples}.
Let $\big((U_i,V_i)\big)_{1\le i\le n}$
be a sequence of i.i.d.~$\mu$-distributed random points, and let $\big((U_{(i)}, V_{(i)})\big)_{1\le i\le n}$ be its a.s.\ unique reordering such that $U_{(1)} < \ldots < U_{(n)}$.
Then, we say that the random permutation $\Perm{V_{(1)}, \ldots, V_{(n)}}$ is sampled from $\mu$, and we let $\sample_n(\mu)$ denote its law.
We often write $\sigma_{n,\mu}$ for an $n$-sized random permutation sampled from $\mu$.
It was shown in \cite{Hoppen_Kohayakawa_Moreira_Rath_MenezesSampaio_2013} that, for any given permuton $\mu$, the sequence of random permutations $(\sigma_{n,\mu})_{n\in\bN}$ converges in probability to the permuton $\mu$, and that the ``pattern densities'' 
$
\prob{\sigma_{n,\mu} = \sigma} ,
$
for $n\in\bN$ and $\sigma\in\fS_n$, characterize the permuton $\mu$.

The permuton limits of Luce random permutations were computed in \cite{Borga_Chatterjee_Diaconis_2025}.
The context is the following:
for each $n\in\bN$, we consider a sequence of positive weights $\bth^{(n)} = (\theta_1^{(n)},\dots,\theta_n^{(n)})$ and define a function $f_n(y) := \theta_{\lfloor yn\rfloor+1}^{(n)}$ for $y\in[0,1)$.

\begin{theorem}[\cite{Borga_Chatterjee_Diaconis_2025}]\label{th: Luce limit permuton}
    For each $n\in\bN$, let $\tau_n$ be a random permutation with law $\Luce{\bth^{(n)}}$.
    Assume that $(f_n)_{n\in\bN}$ converges to some positive, finite, measurable function $f$ almost everywhere on $[0,1)$.
    Then $(\tau_n)_{n\in\bN}$ converges in probability, in the sense of permutons, to the (deterministic) law $\mu_f$ of $\left(\big. U, F(E/f(U)) \right)$, where $U$ and $E$ are independent random variables with laws $\Unif{[0,1]}$ and $\Exp{1}$, and
    \begin{equation}\label{eq: function F}
        F(x) := 1 - \int_0^1 e^{-xf(t)} \mathrm d t \qquad \text{for $x\in[0,\infty)$.}
    \end{equation}
    This permuton $\mu_f$ is characterized by the pattern densities
    \begin{equation}\label{eq: Luce pattern densities}
        \prob{ \sigma_{k,\mu_f} = \sigma }
        = k! \int_{0<u_1<\ldots<u_k<1}
        \prob{
        \tau^{(u_1, \ldots, u_k)}_k = \sigma } \mathrm du_1 \ldots \mathrm du_k\qquad\text{for $k\in\bN$ and $\sigma\in\fS_k$,}
    \end{equation}
     where we let $\tau_k^{(u_1, \ldots, u_k)}$ denote a Luce random permutation with weights $f(u_1), \ldots, f(u_k)$.
    Moreover, the permuton $\mu_f$ has a density $\rho_f$ with respect to $\Leb_{[0,1]^2}$ given by
    \begin{equation}\label{eq: density of the Luce permuton}
        \rho_f(x,y) = \frac{f(x)e^{-f(x)F^{-1}(y)}}{\int_0^1 f(t)e^{-f(t)F^{-1}(y)} \mathrm dt}\qquad \text{for $(x,y)\in(0,1)^2$.}
    \end{equation}
\end{theorem}

\begin{example}
    Take $f: t\mapsto 1-t$, so that $\mu_f$ is the limit permuton of Luce permutations with Sukhatme weights $\theta_i^{(n)} := n-i+1$.
    Then, letting $\varphi:[0,1)\to[0,\infty)$ denote the inverse of $F: x\in[0,\infty) \mapsto \frac{x-1+e^{-x}}{x}$, we get
    \begin{equation}\label{eq: Sukhatme density}
        \rho_f(x,y) = \frac{(1-x)\varphi(y)^2 e^{x\varphi(y)}}{e^{\varphi(y)}-\varphi(y)-1}\qquad\text{for all $(x,y)\in[0,1)^2$.}
    \end{equation}
    See \Cref{fig: illustration permuton convergence} for an illustration.
    As a comparison, see 
    %\Cref{fig: expoly Sukhatme} for the ``expoly-Sukhatme'' function $f(t) = \exp\big[ -\big(\frac{n-k+1}{n}\big)^{-\alpha} \big]$ with $\alpha=1$, and 
    \Cref{fig: Mallows vs LuceMallows} (right) for the ``Mallows-like'' function $f(t) = e^{-\beta t}$ with $\beta=6$.
\end{example}

\begin{figure}
    \centering
    \includegraphics[width=.38\linewidth]{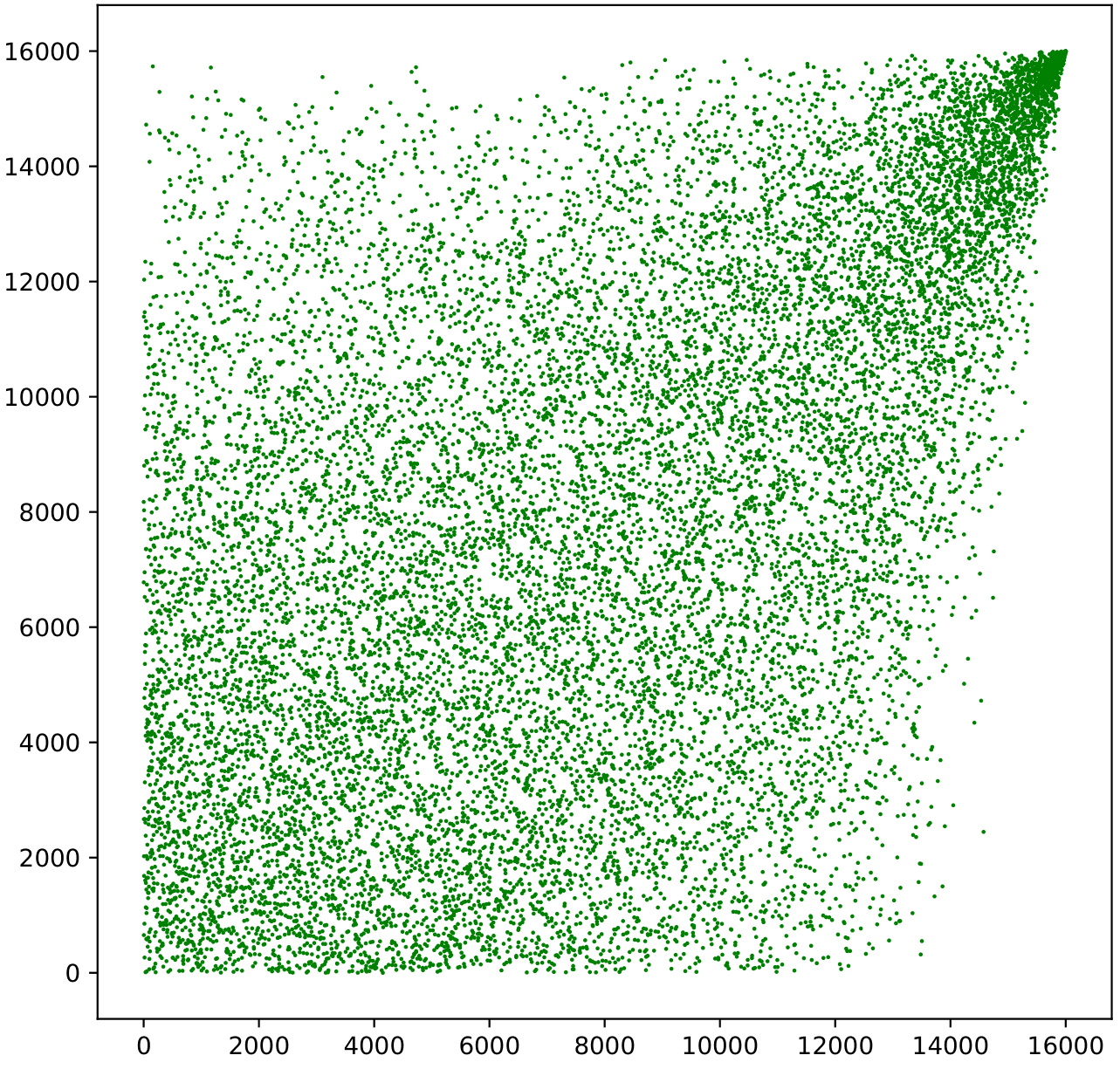}
    \qquad\qquad
    \includegraphics[width=.4\linewidth]{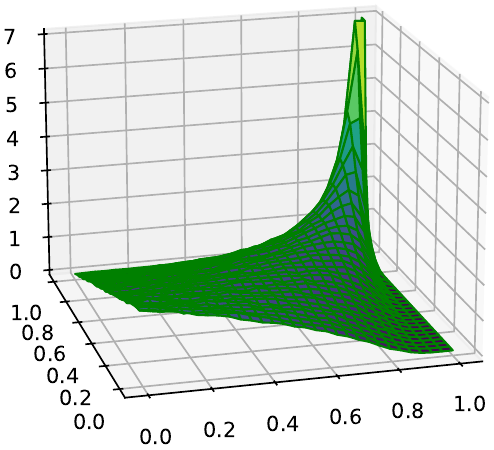}
    \caption{Left: the diagram of a Luce permutation $\tau_n$ with Sukhatme weights $\theta_i^{(n)} = n-i+1$ and size $n=16000$.
    Right: 3D plot of the (truncated) density $\rho_f$ of the ``Sukhatme permuton'' $\mu_f$ with $f(t)=1-t$; see \eqref{eq: Sukhatme density}.}
    \label{fig: illustration permuton convergence}
\end{figure}

% \begin{figure}
%     \centering
%     \includegraphics[width=.38\linewidth]{permpoints_LuceExpolySukhatme_size16000.pdf}
%     \qquad\qquad
%     %\includegraphics[width=.4\linewidth]{heatmap_LuceExpolySukhatme_size_18000.pdf}
%     \caption{Left: a Luce permutation $\tau_n$ with ``expoly-Sukhatme'' weights $\theta_i^{(n)} = \exp\big[ \frac{-n}{n-i+1} \big]$ and size $n=16000$.
%     Right: the (truncated) density $\rho_f$ of the ``expoly-Sukhatme permuton'' $\mu_f$ with $f(t)=\exp\big[ \frac{-1}{1-t} \big]$.}
%     \label{fig: expoly Sukhatme}
% \end{figure}

It should be noted that permuton convergence, in itself, is rather ``weak''.
Indeed, the convergence of a sequence $(\sigma_n)_{n\in\bN}$ towards a permuton $\mu$ implies the convergence of $\varphi(\sigma_n)$ only for a few ``global'' statistics $\varphi : \fS \to \bR$, such as normalized classical pattern counts.
Nonetheless, an emerging line of thought is that permuton convergence, \emph{together with} some ``regularity'' properties, should imply that $\varphi(\sigma_n)$ behaves asymptotically like $\varphi(\sigma_{n,\mu})$, for a wider class of statistics $\varphi$.
One striking example of this is the \emph{equicontinuity} property proposed by Mukherjee in \cite{Mukherjee_2016}, which implies the convergence of small cycle counts.
A similar philosophy can be perceived in the papers \cite{Mueller_Starr_2013}, which relates the LIS of Mallows permutations to their limit permuton, and \cite{Dubach_2025}, which analyzes conjugacy-invariant permutations through a geometric construction akin to that of permuton samples.
The present work can be seen in the continuity of this line of research, as we rely on the ``regularity properties'' of Lemmas~\ref{lem: permuton sample is Luce with random weights} and~\ref{lem: discrete Luce is permuton sample with fixed x-coordinates} to relate the LIS of Luce permutations to their limit permuton (\Cref{th: LLN with variational constant}).
Hence, this also contributes to the following problem raised in \cite{Borga_Chatterjee_Diaconis_2025}:
    ``a natural question is to study how a sequence of exactly distributed Luce permutations behaves differently than the permutations sampled for their limiting permuton''.

\subsection{Statements of the main results and organization of the paper}

In this section we present our main results, with their precise statements for the reader's convenience.

Our first main result concerns the growth rate of the LIS of Luce permutations.
We restrict our attention to sequences of non-increasing weights, which form an important class of weights for the Luce model (see, e.g., the discussions in \cite{chatterjee2024enumerative,Borga_Chatterjee_Diaconis_2025}).

In what follows, if $(a_n)_{n\in\bN}$ and $(b_n)_{n\in\bN}$ are sequences of positive numbers, we use the notation $a_n \asymp b_n$ to signify that the ratios $a_n/b_n$ and $b_n/a_n$ are bounded by an absolute constant.
If $(A_n)_{n\in\bN}$ and $(B_n)_{n\in\bN}$ are sequences of a.s.\ positive random variables, then $A_n \asymp_\bP B_n$ means that $A_n / B_n$ is bounded below and above in probability, in the following sense:
\[
\lim_{\varepsilon\to0} \liminf_{n\to\infty} \prob{ A_n / B_n \in [\varepsilon, 1/\varepsilon] } = 1 \,.
\]

\begin{theorem}[Summary of Theorem~\ref{th: LIS growth rate - general} and Corollary~\ref{cor: LIS is close to expectation if non-increasing weights}]
\label{th: intro growth rate}
    For each $n\in\bN$, let $\bth^{(n)} = (\theta_1^{(n)} \ge \ldots \ge \theta_n^{(n)})$ be a non-increasing sequence of positive weights. %, rescaled so that $\theta_1^{(n)} = 1$.
    For each $k\in\bZ$, %\bN_0$, 
    set
    \[
    B_k^{(n)} := \left\{ i\in[n] : 
    k \le -\log \theta_i^{(n)} < k+1
    %\lfloor -\log \theta_i^{(n)} \rfloor = k
    \right\}
    \quad\text{and}\quad
    m_k^{(n)} := \#{B_k^{(n)}} \,.
    \]
    Then, for $\tau_n$ a $\Luce{\bth^{(n)}}$ random permutation:
    \[
    \expec{\LIS{\tau_n}} \asymp
    \sum_{k\in\bZ} \sqrt{m_k^{(n)}} 
    \quad\text{and}\quad
    \LIS{\tau_n} \asymp_\bP
    \sum_{k\in\bZ} \sqrt{m_k^{(n)}} 
    \,.
    \]
\end{theorem}

\Cref{sec: general growth rate} is dedicated to this result, along with applications and examples.
In \Cref{sec: examples growth rate}, we provide several methods to compute the quantity $\sum_{k\in\bZ} \sqrt{m_k^{(n)}}$ when the weights are of the form $\theta_i^{(n)} = f\big( \frac{i-1}{n} \big)$ for a given function $f$.
In particular, we provide sufficient conditions for the growth rate to be of order $\sqrt n$ or $n^\beta$ for any given $\beta\in[1/2,1)$ (see Example~\ref{ex: growth rates n beta}).
For example, we find that:
\begin{itemize}
    \item For the polynomial Sukhatme weights $\theta_k = {\left( \frac{n-k+1}{n} \right)}^\alpha$ with $\alpha>0$, the LIS grows like $\sqrt n$;
    \item For the expoly-Sukhatme weights $\theta_k = \exp{\left[ - {\left( \frac{n-k+1}{n} \right)}^{-\alpha} \right]}$ with $\alpha>1$, the LIS grows like $n^\beta$ with $\beta = \frac{\alpha}{\alpha+1}$;
    \item For the exponential Sukhatme weights $\theta_k = \alpha^k$ with $\alpha\in(0,1)$, the LIS grows like $n$;
    \item For the Mallows-like weights $\theta_k = q_n^k$ {with $0<q_n\le1$ and} $q_n\to1$, the LIS grows like $\max\big( n\sqrt{1-q_n} \,,\, \sqrt n \big)$.
\end{itemize}

The proof of \Cref{th: LIS growth rate - general} is conducted in \Cref{sec: proof growth rate}.
It partly relies on a ``superuniformity'' property (Proposition~\ref{prop: superuniformity property}) proven in \Cref{sec: superuniformity}.
The latter essentially states that the LIS of any given subpermutation in $\tau_n$, defined by restricting the positions and values to given intervals, stochastically dominates the LIS of a uniformly random permutation.
Finally, we provide in \Cref{sec: concentration} a concentration inequality for the random variable $\LIS{\tau_n}$, implying that the estimate for the expectation in \Cref{th: LIS growth rate - general} also holds in probability; see Corollaries~\ref{cor: McDiarmid for Luce} and~\ref{cor: LIS is close to expectation if non-increasing weights}.

\medskip

Our second main result states a law of large numbers in the $\sqrt n$ regime for ``regular'' weights.
That is, we consider weights of the form $\bth^{(n)}(f) := \big(f\big(\frac{i-1}{n}\big)\big)_{1\le i\le n}$ for some ``sufficiently regular'' function $f : [0,1) \to \bR_+^*$.
We call such weights \emph{canonically associated} with the function $f$.
As a preliminary result, we show via Proposition~\ref{prop: growth rate with phi} and Lemma~\ref{lem: discrete and continuous models share same groth rate} that if $\tau_n$ and $\sigma_{n,\mu_f}$ have laws $\Luce{\bth^{(n)}(f)}$ and $\sample_n(\mu_f)$ respectively, then \Cref{th: intro growth rate} applies with $\sigma_{n,\mu_f}$ instead of $\tau_n$.
In other words, $\LIS{\tau_n}$ and $\LIS{\sigma_{n,\mu_f}}$ share the same growth rate in expectation and in probability.
Further, we prove that in the case $\expec{\LIS{\tau_n}} \asymp \sqrt n$, then both of these random variables are asymptotic to $2 K(f) \sqrt n$ for some constant $K(f)$.
We provide two formulas for this constant, relying on a variational problem; the first one is in terms of the permuton $\mu_f$ (\Cref{th: LLN with variational constant}), while the second one is a more direct reformulation in terms of $f$ (Lemma~\ref{lem: variational reformulate}).
Note that the condition $\expec{\LIS{\tau_n}} \asymp \sqrt n$ can be checked through \Cref{th: LIS growth rate - general}, and specifically, Lemma~\ref{lem: summable condition for sqrt n regime}.

\begin{theorem}[Summary of Theorem~\ref{th: LLN with variational constant} and Lemma~\ref{lem: variational reformulate}]
\label{th: intro LLN}
    Let $f$ be a positive, finite, non-increasing, $\cC^1$ function on $[0,1)$.
    Let $\mu_f$ be the associated permuton of \Cref{th: Luce limit permuton} with density $\rho_f$ given by \eqref{eq: density of the Luce permuton}.
    For each $n\in\bN$, let $\bth^{(n)}(f) := \big( f\big(\frac{i-1}{n}\big) \big)_{1\le i\le n}$ be the canonically associated weights.
    Let $\tau_n$ be a $\Luce{\bth^{(n)}(f)}$ random permutation, and $\sigma_{n,\mu_f}$ be a $\mu_f$-sample.
    Assume that
    {$-\log f$ is convex}
    and that $\expec{\LIS{\tau_n}} \asymp \sqrt n$.
    Then:
    \[
        \lim_{n\to\infty} \frac{1}{\sqrt n} \LIS{\tau_n}
        = \lim_{n\to\infty} \frac{1}{\sqrt n} \LIS{\sigma_{n,\mu_f}} 
        = 2 K(f)
        \quad\text{with}\quad
        K(f) = \sup_{u:[0,1)\nearrow [0,\infty)} \int_0^1 \sqrt{f(t) u'(t) e^{-f(t)u(t)}} \mathrm dt
    \]
    where the limits hold in probability and in expectation, and the supremum ranges over all non-decreasing, right-continuous curves $u : [0,1) \to [0,\infty)$.
\end{theorem}

\Cref{sec: the constant} is dedicated to this result, and spans ideas that we believe to be more meaningful.
We start by showing in \Cref{sec: two models} that the choice of weights $\theta_i^{(n)} = f\big( \frac{i-1}{n} \big)$ enables the ``regularity properties'' of Lemmas~\ref{lem: permuton sample is Luce with random weights} and~\ref{lem: discrete Luce is permuton sample with fixed x-coordinates}.
These two lemmas are formally equivalent, but dual in their interpretation:
Lemma~\ref{lem: permuton sample is Luce with random weights} states that Luce permuton samples $\sigma_{n,\mu_f}$ are distributed like Luce permutations with random weights,
whereas Lemma~\ref{lem: discrete Luce is permuton sample with fixed x-coordinates} states that Luce permutations $\tau_n$ are distributed like Luce permuton samples ``with fixed x-coordinates''.

From there, the general philosophy is to compare the Luce permutation $\tau_n$ with weights $\bth^{(n)}(f)$ (the ``discrete'' model) with the Luce permuton samples $\sigma_{n,\mu_f}$ (the ``continuous'' model).
First, we argue in Lemma~\ref{lem: discrete and continuous models share same groth rate} that both models share the same LIS growth rate.
In order to refine this comparison, 
we construct in \Cref{sec: coupling} a coupling of $\tau_n$ and $\sigma_{n,\mu_f}$, such that they coincide on a ``large'' subpermutation; see Proposition~\ref{prop: coupling discrete and continuous}.
This later allows us to show that $\LIS{\tau_n} / \LIS{\sigma_{n,\mu_f}}$ converges to $1$ in probability, so it suffices to prove \Cref{th: LLN with variational constant} for $\sigma_{n,\mu_f}$.
See the proof of \Cref{th: LLN with variational constant} for details.

Finally, \Cref{sec: proof the constant} contains the remaining ingredients to prove \Cref{th: LLN with variational constant}.
The main input is the famous law of large numbers obtained by Deuschel and Zeitouni in \cite{Deuschel_Zeitouni_1995} for the LIS of permuton samples;
after settling some technical issues, we apply their result to the continuous model $\sigma_{n,\mu_f}$, and conclude that it also holds for the discrete model $\tau_n$ thanks to our coupling.
Finally, in \Cref{sec: variational problem} we discuss the computation of the constant $K(f)$, provide a simpler formula (Lemma~\ref{lem: variational reformulate}), and present some bounds.

\subsection{Notation}
\label{sec: notation}

Throughout, we let $\bN := \{1, 2, \ldots\}$ and $\bN_0 := \{0, 1, 2, \ldots\}$.
% The letter $n$ is always used to designate a number in $\bN$, eventually meant to go to $\infty$.
For $n\in\bN_0$, we let $[n] := \{1, \ldots, n\}$ with convention $[0]=\emptyset$.
We extend this notation by letting $[x] := [ \lfloor x \rfloor ]$ for $x\ge0$.
If $(u_n)_{n\in\bN}$ and $(v_n)_{n\in\bN}$ are two sequences of positive numbers, we write $u_n = \cO(v_n)$ if $(u_n / v_n)_{n\in\bN}$ is bounded, and $u_n = o(v_n)$ if $u_n / v_n \to 0$ as $n\to\infty$.
If we have both $u_n = \cO(v_n)$ and $v_n = \cO(u_n)$, we write $u_n = \Theta(v_n)$.

If $I$ is an arbitrary index set and $(u_\alpha)_{\alpha\in I}$, $(v_\alpha)_{\alpha\in I}$ are two families of numbers, $u_\alpha \asymp v_\alpha$ means that there exists an absolute constant $C>0$ such that $C^{-1} \le u_\alpha / v_\alpha \le C$ holds for all $\alpha\in I$.
In this paper, when a function $f : [0,1) \to (0,\infty)$ is involved, the \enquote{absolute} constant $C$ is allowed to depend on $f$.

For $n\in\bN$, we let $\fS_n$ denote the symmetric group on $[n]$, and we let $\fS := \bigsqcup_{n\in\bN} \fS_n$ denote the set of all permutations.
If $\sigma\in\fS_n$ and $i\in[n]$, we freely write $\sigma(i)$ or $\sigma_i$, although $\sigma_n$ usually denotes a permutation of size $n$.
This conflict of notation should cause no confusion.
We may use the \emph{one-line notation} $\sigma = \sigma_1 \sigma_2 \ldots \sigma_n$.
% We shall usually stick to the following convention for random permutations:
% $\sigma_n$ denotes a uniformly random permutation in $\fS_n$; $\sigma_{n,\mu}$ a $\mu$-sample for a given permuton $\mu$; and $\tau_n$ a Luce random permutation with given weights $\bth^{(n)}$.
% If $f$ is a positive, finite function on $[0,1)$, we let $\bth^{(n)} := \big( f\big( \frac{i-1}{n} \big) \big)_{1\le i\le n}$ denote the \enquote{canonically} associated weights.

If $y_1, \ldots, y_n$ is a sequence of pairwise distinct numbers, we let $\LIS{y_1, \ldots, y_n}$ denote the length of their longest increasing subsequences.
In particular, if $\sigma\in\fS_n$, then $\LIS{\sigma} = \LIS{\sigma_1, \ldots, \sigma_n}$.
If $\cP$ is a finite point set on $\bR^2$ with pairwise distinct x-coordinates and pairwise distinct y-coordinates, we let $\Perm{\cP}$ denote the permutation that it induces.
That is, $\Perm{\cP} := \Perm{Y_{(1)}, \ldots, Y_{(n)}}$ where $\big( X_{(i)}, Y_{(i)} \big)_{i\ge 1}$ are the points of $\cP$ ordered by increasing x-coordinate.
Then, we let $\LIS{\cP} := \LIS{\Perm{\cP}} = \LIS{Y_{(1)}, \ldots, Y_{(n)}}$.

If $X,Y$ are two real random variables, we say that $X$ \emph{stochastically dominates} $Y$ if, for all $t\in\bR$,
$
\prob{X\ge t} \ge \prob{Y\ge t} .
$
Equivalently, $X$ stochastically dominates $Y$ if and only if there exists a coupling $(\tilde X, \tilde Y)$ for which $\tilde X \ge \tilde Y$ holds almost surely.
If $(X_n)_{n\in\bN}$ is a sequence of real random variables and $x\in \bR$, we say that
$
\limsup X_n \le x 
$
in probability if, for any given $\delta>0$,
$\prob{ X_n \ge x + \delta } \to 0$ as $n\to\infty$.
Likewise, we say that $\liminf X_n \ge x$ in probability if, for any given $\delta>0$, $\prob{ X_n \le x - \delta } \to 0$ as $n\to\infty$.

\subsection{Discussion and open questions}
\label{sec: discussion}

\paragraph{A comparison with the Mallows model.}
Let $\tau$ be $\Luce{\theta_1, \dots, \theta_n}$ distributed.
Then, the sequence of values $\tau^{-1}(1), \ldots, \tau^{-1}(n)$ can be constructed by drawing balls from an urn without replacement:
first, $\tau^{-1}(1)$ is chosen among $[n]$ where each integer $k$ has a probability proportional to $\theta_k$ of being picked.
If, say, $\tau^{-1}(1)=\ell$, then $\tau^{-1}(2)$ is chosen among $[n] \setminus \{\ell\}$ where each remaining integer $k$ has a probability proportional to $\theta_k$ of being picked, and so on.
Such a construction is called a \emph{$p$-biased} distribution (with $p_k := \theta_k$);
see, e.g., \cite{Pitman_Tang_2019}.

In comparison, the Mallows distribution on $\fS_n$ with parameter $q \ge 0$ can be defined as follows.
First, $\pi(1)$ is chosen among $[n]$ where each integer $k$ has a probability proportional to $q^k$ of being picked.
If, say, $\pi(1) = \ell$, then $\pi(2)$ is chosen among $[n]\setminus\{\ell\}$ where each integer $k$ has a probability proportional to $q^{\mathrm{rk}(k)}$ of being picked, where $\mathrm{rk}$ is the unique increasing bijection from $[n]\setminus\{\ell\}$ to $[n-1]$.
The Mallows distribution is the law of $\pi$ (or, equivalently, the law of $\pi^{-1}$).
Such a construction is called a \emph{$p$-shifted} distribution (with $p_k := q^k$);
see, again, \cite{Pitman_Tang_2019}.

Therefore, one might wish to compare the Mallows model with parameter $q_n > 0$ to the Luce model with weights $\theta_k^{(n)} := q_n^k$.
Such a comparison is e.g.\ proposed in \cite{Pinsky_2026}, in the context of the secretary problem.
For our purposes, write $q_n = e^{-u_n}$ and assume that $u_n > 0$.
Then, the LIS of Mallows permutations $\pi_n$ with parameter $q_n$ was analyzed in \cite{Mueller_Starr_2013,Bhatnagar_Peled_2015,Basu_Bhatnagar_2017}: 
if $n u_n \to 0$ then $\expec{\LIS{\pi_n}} \asymp \sqrt{n}$;
if $u_n \to 0$ and $\liminf n u_n > 0$ then $\expec{\LIS{\pi_n}} \asymp n\sqrt{u_n}$; 
and if $\liminf u_n > 0$ then $\expec{\LIS{\pi_n}} \asymp n$.
Using \Cref{th: LIS growth rate - general}, we compute in Example~\ref{ex: Mallows-like weights LIS growth rate} the LIS growth rate for the Luce model with weights $\theta_k^{(n)} := q_n^k = e^{-ku_n}$;
interestingly, we find the same growth rates as for the Mallows model!

Going further, $(\pi_n)_{n\in\bN}$ admits a permuton limit with a density if and only if $nu_n \to \beta$ for some $\beta\in[0,\infty)$, in which case the limit permuton $\mu_\beta^{\rm M}$ has a smooth bounded density \cite{Starr_2009}.
Likewise, \Cref{th: Luce limit permuton} applies for $(\tau_n)_{n\in\bN}$ under the same condition, in which case the limit permuton is $\mu_{f_\beta}$ with $f_\beta: t\mapsto e^{-\beta t}$.
See \Cref{fig: Mallows vs LuceMallows} for an illustration:
although the two densities are clearly different, they share some similarities.
On the one hand, it was shown in \cite{Mueller_Starr_2013} that $\LIS{\pi_n}$ is asymptotic to $2 {\cal L}_\beta \sqrt{n}$ where
\[
    \mathcal{L}_\beta = \beta^{-1/2} \mathrm{arsinh}\big( \sqrt{e^\beta-1} \big) \,.
\]
On the other hand, if $\theta_k^{(n)} := e^{-k\beta/n}$, we know that $\LIS{\tau_n}$ is asymptotic to $2 K(f_\beta) \sqrt{n}$ for the constant $K(f_\beta)$ given by \Cref{th: LLN with variational constant} and Lemma~\ref{lem: variational reformulate}.
We discuss the computation of $K(f_\beta)$ in \Cref{sec: variational problem}, and show that $K(f_\beta) \ne \mathcal{L}_\beta$ in general.

\begin{figure}
    \centering
    \includegraphics[width=0.4\linewidth]{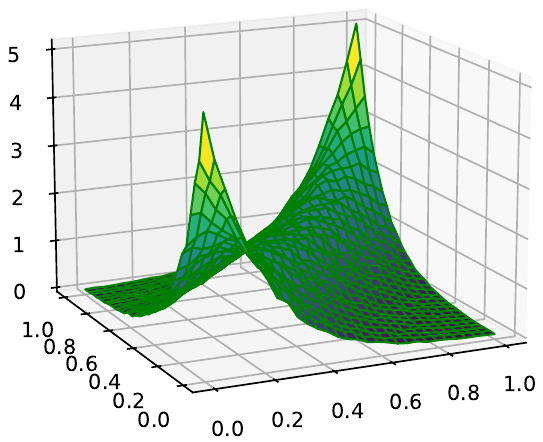}
    \qquad\qquad
    \includegraphics[width=0.4\linewidth]{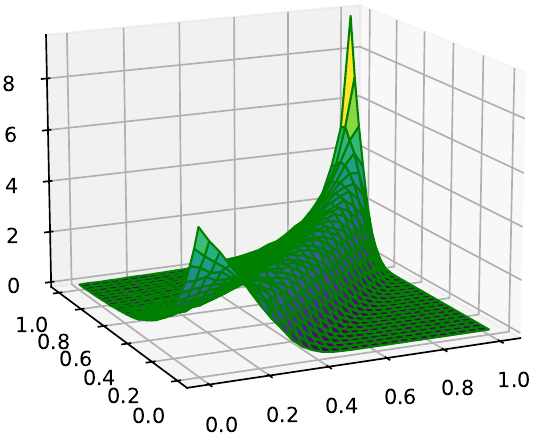}
    \caption{The 3D plots of two Mallows-related permutons with $\beta=6$:
    the Mallows permuton $\mu_\beta^{\rm M}$
    (left), and
    the ``Mallows-like'' Luce permuton $\mu_{f_\beta}$
    (right).
    }
    \label{fig: Mallows vs LuceMallows}
\end{figure}

\paragraph{Assumptions of the weights.}
Our main results concerning the LIS focus on non-increasing sequences of weights $(\theta_i)_{1\le i\le n}$, and we often assume that they are of the form $\theta_i = f((i-1)/n)$ for some function $f$.
The rigidity of the latter choice is discussed in \Cref{rem: functions that diverge at 0} and throughout \Cref{sec: two models}.

Non-monotone weights seem difficult to study, but one might wish to consider \emph{non-decreasing} sequences of weights.
Studying the LIS of Luce permutations with non-decreasing weights is then equivalent to studying the LDS (longest \emph{decreasing} subsequence) of Luce permutations with non-increasing weights.
The ``superuniformity'' property of Proposition~\ref{prop: superuniformity property} is easily adapted to show that the LDS of Luce permutations with such weights are stochastically dominated by that of uniform permutations, and thus, at most $2\sqrt n$.
For weights of the form $\bth^{(n)}(f)$, the proof of \Cref{th: LLN with variational constant} can be readily adapted to show that $\LDS{\tau_n}$ is asymptotic to $K^*(f) \sqrt n$ for the appropriate constant $K^*(f)$.
For other types of weights, we believe that the technique behind \Cref{th: LIS growth rate - general} can also be adapted to compute the growth rate of the LDS, but we have chosen not to pursue this any further.

\paragraph{Future prospects.}
A burning question is that of the fluctuations of $\LIS{\tau_n}$.
Unfortunately, these seem largely out of reach for the moment.
Let us note that currently, the fluctuations of the LIS are known for very few non-uniform distributions.

On another note, we believe that our coupling approach (cf \Cref{sec: two models,sec: coupling}) can be fruitful in the study of statistics other than the LIS.
For example, can we use a similar approach to study the number of fixed points of $\tau_n$?
This does not seem particularly clear or likely (see, e.g., the discussion in \cite[Section~5.2]{Borga_Chatterjee_Diaconis_2025}), but we intend to pursue this line of research.

\section{The growth rate of the LIS}
\label{sec: general growth rate}

\subsection{Statement and examples}
\label{sec: examples growth rate}

This \Cref{sec: general growth rate} is dedicated to the growth rate of the LIS of Luce random permutations.
We will focus on non-increasing weights and prove the following:

\begin{theorem}\label{th: LIS growth rate - general}
    Let $n\in\bN$ and $\bth^{(n)} = (\theta_1^{(n)} \ge \ldots \ge \theta_n^{(n)})$ be a non-increasing sequence of positive weights.
    For each $k\in\bZ$, set
    \[
    B_k^{(n)} := \left\{ i\in[n] : k\le -\log \theta_i^{(n)} < k+1 \right\}
    \quad\text{and}\quad
    m_k^{(n)} := \#{B_k^{(n)}} \,.
    \]
    Then, for $\tau_n$ a $\Luce{\bth^{(n)}}$ random permutation:
    \[
    \expec{\LIS{\tau_n}} \asymp
    \sum_{k\in\bZ} \sqrt{m_k^{(n)}} \,.
    \]
\end{theorem}

\begin{remark}
    Since $\sum_{k\in\mathbb Z}m_k^{(n)}=n$, it is easily seen that $\sum_{k\in\bZ} \sqrt{m_k^{(n)}} \ge \sqrt{n}$.
    Even better, Corollary~\ref{cor: Luce LIS is at least 2 sqrt n} shows that $\expec{\LIS{\tau_n}}$ is at least $2\sqrt n$ asymptotically.
\end{remark}

The proof crucially relies on the exponential representation of Luce permutations (Lemma~\ref{lem: exponential representation}).
Then, the upper bound is quite simple.
It uses an elementary partitioning technique, along with Lemma~\ref{lem: upper bound on density gives upper bound on LIS}, which provides an upper bound for the LIS of random variables that are independent but not necessarily identically distributed.
The lower bound requires more work, and a finer partitioning.
On each subset of our partitioning, we use a ``superuniformity'' property of the Luce model:
under the assumption that the weights are non-increasing, if $I$ and $J$ are two intervals in $[n]$, then the LIS of the subpermutation $\{ \tau_n(i) : i\in I, \tau_n(i) \in J \}$ stochastically dominates the LIS of a uniformly random permutation of same size.
This property, formulated slightly differently in Proposition~\ref{prop: superuniformity property} in terms of continuous random variables, is proved in \Cref{sec: superuniformity}.

\medskip

Before moving on with the proofs, let us present some applications of \Cref{th: LIS growth rate - general}.
We start with the computation of the LIS growth rate for Mallows-like weights.

\begin{example}\label{ex: Mallows-like weights LIS growth rate}
    Consider the Mallows-like weights $\theta_k^{(n)} = e^{-ku_n}$ where $u_n > 0$.
    Then, for $0\le k\le nu_n$, $m_k^{(n)}$ is the number of positive integers at most $n$ in the interval $ [k/u_n , (k+1)/u_n)$ of length $1/u_n$.
    \begin{itemize}
        \item Assume $\liminf u_n > 0$.
        Then, all $m_k^{(n)}$'s are bounded.
        Since $\sum_{k\ge0} m_k^{(n)} = n$, we get that $\sum \sqrt{m_k^{(n)}} \asymp n$.
        \item Assume $nu_n \to 0$.
        Then, for $n$ large enough, $m_0^{(n)} = n$.
        Therefore, $\sum \sqrt{m_k^{(n)}} \asymp \sqrt n$.
        \item Assume $u_n \to 0$ and $\liminf nu_n > 0$.
        {For every $k\ge0$, we have $m_k^{(n)}\le 1/u_n+1$, and at most $\lfloor nu_n\rfloor+1$ of these counts are nonzero. Consequently,
        $\frac{n}{\sqrt{1/u_n+1}}
            \le \sum_{k\ge0}\sqrt{m_k^{(n)}}
            \le \sqrt{n\bigl(\lfloor nu_n\rfloor+1\bigr)}$.
        Indeed, the lower bound follows by summing
        $m_k^{(n)}\le \sqrt{1/u_n+1}\sqrt{m_k^{(n)}}$,
        while the upper bound follows from the Cauchy--Schwarz inequality and $\sum_{k\ge0}m_k^{(n)}=n$.}
        We deduce that $\sum \sqrt{m_k^{(n)}} \asymp n \sqrt{u_n}$.
    \end{itemize}
\end{example}

Now, we focus on the following important class of weights:
we are given a positive, finite, non-increasing function $f$ on $[0,1)$, 
normalized so that $f(0)=1$, 
and we choose the ``canonically associated'' weights $\bth^{(n)}(f) := \big( f\big(\frac{i-1}{n} \big) \big)_{1\le i\le n}$.
In what follows, the constants hidden by the $\asymp$ symbols may depend on the function $f$.
 
\begin{proposition}\label{prop: growth rate with phi}
    Consider the weights $\bth^{(n)}(f)$, and let $m_k^{(n)}$ be as in \Cref{th: LIS growth rate - general}.
    For each $k\in\bN_0$, set
    \begin{equation*}
        A_k := \{ x\in(0,1) : k\le -\log f(x) < k+1 \}
        \quad\text{and}\quad
        p_k := \Leb(A_k) \,.
    \end{equation*}
    {Assume that the sequence $(p_k)_{k\ge0}$ is non-increasing (this holds, e.g., if $-\log f$ is convex).}
    Then:
    \begin{equation*}
        \sum_{k\in\bN_0} \sqrt{m_k^{(n)}}
        \asymp \sum_{k\in\bN_0} \min\{np_k , \sqrt{np_k} \}
        =: \psi_f(n) \,.
    \end{equation*}
\end{proposition}

\begin{proof}
    With our definition, $m_k^{(n)}$ is simply the number of integers $i\in[0,n-1]$ such that $i/n$ lies in the interval $A_k$.
    This interval has length $p_k$, so 
    \begin{equation}\label{eq: bound m_k np_k}
        np_k - 1 \le m_k^{(n)} \le np_k + 1 \,.
    \end{equation}
    Let  $K_n \in \bN_0$ be such that $p_k\ge 2/n$ for $k \le K_n$ and $p_k<2/n$ for $k > K_n$.
    For $k\le K_n$, thanks to \eqref{eq: bound m_k np_k}, $\frac12 np_k \le m_k^{(n)} \le 2np_k$, and thus:
    \begin{equation}\label{eq: comparison on big p_k}
        \frac{1}{2} \sum_{0\le k\le K_n} \min\{np_k , \sqrt{np_k} \}
        \le \sum_{0\le k\le K_n} \sqrt{m_k^{(n)}}
        \le 2 \sum_{0\le k\le K_n} \min\{np_k , \sqrt{np_k} \} \,.
    \end{equation}
    On the other hand, for $k> K_n$, it holds that $np_k < 2$ and by \eqref{eq: bound m_k np_k} that $m_k^{(n)} \le 3$, and thus:
    \begin{equation}\label{eq: linear approx for small parameters}
        \frac12 np_k \le \min\{ np_k, \sqrt{np_k} \} \le np_k
        \quad\text{and}\quad
        \frac12 m_k^{(n)} \le \sqrt{m_k^{(n)}} \le m_k^{(n)} \,.
    \end{equation}
    Hence, it remains only to argue that
    \begin{equation*}
        \sum_{k> K_n} m_k^{(n)}
        \simeq \sum_{k> K_n} np_k 
    \end{equation*}
    in the right sense.
    However, observe that the above quantities are equal to
    \begin{equation*}
        \card{ i\in[n] : -\log f\big(\tfrac{i-1}{n}\big) \geq K_n+1} \,,
        \quad\text{resp.}\quad
        n\Leb\big(\big\{ x\in(0,1) : -\log f(x) \geq K_n+1 \big\}\big) \,.
    \end{equation*}
    Thanks to the monotony of $f$, the set $\big\{ x\in(0,1) : -\log f(x) \geq K_n+1 \big\}$ coincides (up to a set of null measure) with an interval $(1-\varepsilon_n, 1)$, and the left-hand side above is the number of coordinates $(i-1)/n$ in this interval.
    Therefore:
    \begin{equation}\label{eq: comparison on small p_k}
        \abs{
        \sum_{k> K_n} m_k^{(n)}
        - \sum_{k> K_n} np_k 
        } \le 2 \,.
    \end{equation}
    Using \cref{eq: comparison on big p_k,eq: linear approx for small parameters,eq: comparison on small p_k}, it is easy to conclude the proof.
\end{proof}

\begin{remark}\label{rem: functions that diverge at 0}
    One could wish to consider functions $f$ on $(0,1)$ that are allowed to diverge at $0$, tweaking the weights to be $\theta_i := f\big(\frac{i}{n+1}\big)$.
    There would be no obstacle in extending the proof of Proposition~\ref{prop: growth rate with phi} to include such functions $f$, the main modifications being that the indices $k$ should be taken in $\bZ$ and that $(p_{-k})_{k\ge0}$ should be non-increasing for large $k$.
    However, there would be no real gain in generality:
    since the set $A_{-k}^{(f)}$ associated with $f$ (almost) coincides with the set $A_{k-1}^{(1/f)}$ associated with $1/f$, it suffices to study the behaviors of both $x\mapsto f(x)$ and $x\mapsto f(1-x)^{-1}$ close to $1$.
    More precisely, letting 
    \[
        f_1(x) := \min\{ 1 , f(x) \}
        \quad\text{and}\quad
        f_2(x) := \min\left\{ 1 , \frac{1}{f(1-x)} \right\} \,,
    \]
    and $m_k^{(n)}$ be defined in terms of the weights $\theta_i^{(n)} := f\big(\frac{i}{n+1}\big)$, one could show that if $f_1$ and $f_2$ both satisfy the assumptions of Proposition~\ref{prop: growth rate with phi}, then:
    \begin{equation*}
        \sum_{k\in\bZ} \sqrt{m_k^{(n)}} \asymp \psi_{f_1}(n) + \psi_{f_2}(n) \,.\qedhere
    \end{equation*}
\end{remark}

\begin{remark}
    If $f$ is continuous and strictly decreasing on $(0,1)$, then we can write $p_k = f^{-1}(e^{-(k+1)}) - f^{-1}(e^{-k})$ for each $k\in\bN_0$ such that these inverses exist.
    They do exist for all $k\in\bN_0$ whenever $\inf f = 0$.
    Otherwise, if $f$ is uniformly bounded below on $[0,1)$, then it is easy to see that $\psi_f(n) \asymp \sqrt{n}$.
    Indeed, in that case, $m_k^{(n)} = 0$ for all $k > \sup (-\log f)$.
\end{remark}

\begin{remark}
    In the present context, if $\tau_n$ is a Luce random permutation with weights $\bth^{(n)}(f)$, then $\expec{\LIS{\tau_n}} = o(n)$ always holds regardless of the function $f: [0,1) \to (0,\infty)$.
    This is a simple consequence of the permuton convergence of \Cref{th: Luce limit permuton}, and the fact that the limit permuton $\mu_f$ has a density;
    see, e.g., \cite[Proposition~2.3]{Bassino_Bouvel_Drmota_Feray_Gerin_Maazoun_Pierrot_2022} or \cite[Proposition~2.1]{Dubach_2024}.
    In contrast, if one considers e.g.\ the ``exponential Sukhatme'' weights $\theta_k := \alpha^{-k}$ for some $\alpha>1$, then \Cref{th: Luce limit permuton} does not apply and \Cref{th: LIS growth rate - general} gives that $\expec{\LIS{\tau_n}} \asymp n$.
\end{remark}

Now, let us give criteria for $\psi_f(n)$ to be of order $\sqrt n$ or $n^\beta$ for some $\beta \in (1/2, 1)$.

\begin{lemma}\label{lem: summable condition for sqrt n regime}
    %Let $f$ be a positive, finite, measurable function on $[0,1)$.
    Let $(p_k)_{k\ge0}$ and $n\in\bN \mapsto \psi_f(n)$ be as in Proposition~\ref{prop: growth rate with phi}.
    Then:
    \begin{equation*}
        \psi_f(n) = \cO(\sqrt{n})
        \quad\text{ if and only if }\quad
        \sum_{k\ge0} \sqrt{p_k} < \infty \,.
    \end{equation*}
\end{lemma}

\begin{proof}
    First, assume that the family $(\sqrt{p_k})_{k\ge0}$ is summable.
    Then, recalling the definition of $\psi_f(n)$:
    \begin{equation*}
        \psi_f(n) = \sum_{k\ge0} \min\{ np_k , \sqrt{np_k} \}
        \le \sqrt{n} \sum_{k\ge0} \sqrt{p_k} 
        = \cO(\sqrt n) \,.
    \end{equation*}
    Now, assume instead that $\psi_f(n) = \cO(\sqrt n)$.
    Fix $K\in\bN$, and take $n_0 = n_0(K)$ large enough so that
    \begin{equation*}
        \text{for all } n\ge n_0,\quad
        \text{for all } k\in[0,K] \text{ such that $p_k > 0$},\quad
        np_k \ge 1 \,,
    \end{equation*}
    so that $\min\{ np_k , \sqrt{np_k} \} = \sqrt{np_k}$ for $n\ge n_0$ and $k\in[0,K]$.
    This gives
    \begin{equation*}
        \sum_{k=0}^K \sqrt{p_k} \le \frac{\psi_f(n)}{\sqrt{n}} \le C
    \end{equation*}
    for some absolute constant $C> 0$.
    Letting $K\to\infty$, this concludes the proof.
\end{proof}

\begin{example}\label{ex: Sukhatme}
    Thanks to Proposition~\ref{prop: growth rate with phi} and Lemma~\ref{lem: summable condition for sqrt n regime}, we can easily find functions $f$ for which the growth rate of the LIS is $\sqrt n$.
    \begin{itemize}
        \item Let $\alpha>0$ and $f: x\in[0,1] \mapsto (1-x)^\alpha$.
        The function $-\log f$ is convex.
        We have $p_k = e^{-k/\alpha} \big( 1 - e^{-1/\alpha} \big)$ for $k\ge0$, so $\sum_k \sqrt{p_k} < \infty$.
        As an application, Luce random permutations with Sukhatme weights $\theta_i = n-i+1$ and polynomial Sukhatme weights $\theta_i = (n-i+1)^\alpha$ have an LIS growth rate of order $\sqrt n$.
        \item Let $\alpha>0$ and $f: x\in[0,1] \mapsto \exp{\left[1-(1-x)^{-\alpha} \right]}$.
        The function $-\log f$ is convex.
        We have $p_k 
        %= k^{-1/\alpha} - (k+1)^{-1/\alpha} 
        \asymp k^{-\frac{\alpha+1}{\alpha}}$ for $k\ge 1$, so $\sum_k \sqrt{p_k} < \infty$ if and only if $\alpha < 1$.
        As an application, Luce random permutations with expoly-Sukhatme weights $\theta_i = \exp{\left[ -\big( \frac{n}{n-i+1} \big)^\alpha \right]}$ have an LIS growth rate of order $\sqrt n$ if and only if $\alpha < 1$.
        % \item Let $f: x\in(0,1] \mapsto -\log x$.
        % The function $-\log f$ is concave on $(0,e^{-1})$ and convex on $(e^{-1}, 1)$.
        % We have $p_k = \exp\big[ -e^{-(k+1)} \big] - \exp\big[ -e^{-k} \big]$ for each $k\in\bZ$, so $p_k \asymp e^{-k}$ for $k\in\bN$ and $p_{-k} \sim \exp\big[ -e^{k-1} \big]$ as $k\to\infty$.
        % It is easy to check that $\sum_{k\in\bZ} \sqrt{p_k} < \infty$, and thus, Luce random permutations with weights $\theta_i = \log\big(\frac{n+1}{i}\big)$ have an LIS growth rate of order $\sqrt{n}$.
    \end{itemize}
\end{example}

\begin{example}\label{ex: growth rates n beta}
    Assume that $p_k \asymp k^{-\gamma}$ for some $\gamma>1$.
    In that case, it is not difficult to directly compute
    \[
    \psi_f(n) \asymp 
    \begin{cases}
        \sqrt n &\text{ if } \gamma>2;\\
        \sqrt n \log n &\text{ if } \gamma=2;\\
        n^{1/\gamma} &\text{ if } \gamma<2.
    \end{cases}
    \]
    In particular, for $f: x\in[0,1] \mapsto \exp{\left[1-(1-x)^{-\alpha} \right]}$, we see that the LIS growth rate is $\sqrt n$ if $\alpha < 1$; $\sqrt n \log n$ if $\alpha = 1$; and $n^{\frac{\alpha}{\alpha+1}}$ if $\alpha > 1$.
\end{example}

Finally, we conclude with another formula which can be convenient for the computation of $\psi_f(n)$, in the case of smooth functions.

\begin{lemma}
    Let $f$ be a finite, strictly decreasing function on $[0,1]$, positive on $[0,1)$, with $f(0)=1$ and $f(1)=0$.
    Assume that $f$ is $\cC^1$ on $[0,1]$ and that $H := -\log f$ is convex.
    Then, with the notation of Proposition~\ref{prop: growth rate with phi}:
    \begin{equation*}
        p_k \le
        \frac{1}{H' \circ H^{-1}(k)}
        =
        \frac{-e^{-k}}{f' \circ f^{-1}(e^{-k})} 
        \le p_{k-1} 
    \end{equation*}
    for all $k\in\bN$.
\end{lemma}

\begin{proof}
    Observe that, thanks to our assumptions, $H$ is one-to-one from $(0,1)$ on $(0,\infty)$ and $H^{-1}$ is $\cC^1$ on $(0,1)$.
    For $k\in\bN$, we have $p_k = H^{-1}(k+1) - H^{-1}(k)$.
    By the mean value theorem, there exists $\xi_k \in [k,k+1]$ such that 
    \[
    p_k = (H^{-1})'(\xi_k) = \frac{1}{H' \circ H^{-1}(\xi_k)} \,.
    \]
    Since $H$ is increasing and convex, we can bound
    \[
    \frac{1}{H' \circ H^{-1}(k+1)} \le
    p_k \le \frac{1}{H' \circ H^{-1}(k)} \,,
    \]
    and the conclusion follows.
\end{proof}

\subsection{A superuniformity property}
\label{sec: superuniformity}

\begin{lemma}\label{lem: increasing ratio of exponentials}
    Let $0\le u\le v\le \infty$ with $v>0$.
    The function 
    \[G : \theta\in[0,\infty) \to \frac{1-e^{-\theta u}}{1-e^{-\theta v}} \,,\] 
    with convention 
    $e^{-\theta \times \infty} = 0$,
    is non-decreasing.
\end{lemma}

\begin{proof}
    For $v=\infty$, this is immediate.
    For $v<\infty$, the derivative has the same sign as $ue^{\theta v} - ve^{\theta u} + (v-u)$.
    The latter is $0$ at $\theta=0$ and has derivative $uv(e^{\theta v} - e^{\theta u}) \ge 0$.
\end{proof}

\begin{proposition}\label{prop: superuniformity property}
    Let $\theta_1\ge \dots\ge \theta_k$ be a non-increasing sequence of positive weights, and fix $0\le c < d \le \infty$.
    Independently over $i\in[k]$, let $Z_i$ be a random variable distributed like $E_i/\theta_i$ conditioned to be in $[c,d]$, where $E_i\sim\Exp{1}$.
    Then,
    $\LIS{Z_1, \dots, Z_k}$ stochastically dominates $\LIS{\sigma_k}$ where $\sigma_k \sim \Unif{\fS_k}$.
\end{proposition}

\begin{proof}
    If $d=\infty$, the subsequent formulas are to be understood with the usual conventions.
    For $i\in[k]$, let $G_i$ denote the distribution function of $Z_i$:
    \[
        \text{for } t\in[c,d],\quad
        G_i(t) = \prob{Z_i \le t}
        = \frac{\int_{c}^t \theta_i e^{-s\theta_i} \mathrm{d}s}{\int_{c}^{d} \theta_i e^{-s\theta_i} \mathrm{d}s}
        = \frac{e^{-c\theta_i} - e^{-t\theta_i}}{e^{-c\theta_i} - e^{-d\theta_i}}
        = \frac{1 - e^{-(t-c)\theta_i}}{1 - e^{-(d-c)\theta_i}} \,.
    \]
    By Lemma~\ref{lem: increasing ratio of exponentials}, we know that for any fixed $t\in(c,d)$:
    \begin{equation}\label{eq: decreasing distribution function}
        \text{for all } 1\le i< j\le k,\quad
        G_i(t) \ge G_j(t) \,.
    \end{equation}
    Now, let $U_1, \dots, U_k$ be i.i.d.\ random variables uniformly distributed on $[0,1]$, so that $(Z_1, \dots, Z_k)$ is distributed like $\big( G_1^{-1}(U_1), \dots, G_k^{-1}(U_k) \big)$.
    We claim that 
    \begin{equation}\label{eq: claim coupling LIS CDF}
        \LIS{U_1, \dots, U_k} \le \LIS{G_1^{-1}(U_1), \dots, G_k^{-1}(U_k)} \,.
    \end{equation}
    Indeed, if $i<j$ are such that $U_i < U_j$, then
    \begin{align*}
        G_i^{-1}(U_i) < G_i^{-1}(U_j) = G_i^{-1}(G_j(G_j^{-1}(U_j))) \le G_i^{-1}(G_i(G_j^{-1}(U_j))) = G_j^{-1}(U_j) \,,
    \end{align*}
    where we used that $t \mapsto G_i(t)$ is increasing for the first inequality, and \eqref{eq: decreasing distribution function} for the second inequality.
    This proves \eqref{eq: claim coupling LIS CDF}, and thus, that $\LIS{Z_1, \dots, Z_k}$ stochastically dominates $\LIS{U_1, \dots, U_k}$.
    We conclude by recalling that the random variable $\LIS{U_1, \dots, U_k}$ is distributed like the LIS of a uniformly random permutation.
\end{proof}

% It will be useful to reformulate this result in a context closer to \Cref{th: Luce limit permuton}:

% \begin{corollary}\label{cor: superuniformity property continuous}
%     Let $f:[0,1) \to \bR_+^*$ be a measurable non-increasing function, and let $F$ be defined by \eqref{eq: function F}.
%     Fix $0\le a<b \le1$ and $0\le x_1< \dots< x_k< 1$.
%     Independently over $i\in[k]$, let $Y_i$ be a random variable distributed like $F(E/f(x_i))$ conditioned to be in $[a,b]$, where $E\sim\Exp{1}$.
%     Then:
%     $\LIS{Y_1, \dots, Y_k}$ stochastically dominates $\LIS{\sigma_k}$ where $\sigma_k \sim \Unif{\fS_k}$.
% \end{corollary}

% \begin{proof}
%     Write $(c, d) := (F^{-1}(a), F^{-1}(b))$.
%     For $i\in[k]$, let $Z_i := F^{-1}(Y_i)$ and note that it is distributed like $E/f(x_i)$ conditioned to be in $[c, d]$.    
%     Thanks to Proposition~\ref{prop: superuniformity property}, we know that  $\LIS{Z_1, \dots, Z_k}$ stochastically dominates $\LIS{\sigma_k}$.
%     Since $F$ is increasing, it holds that $\LIS{Z_1, \dots, Z_k} = \LIS{Y_1, \dots, Y_k}$, concluding the proof.
% \end{proof}

\begin{corollary}\label{cor: Luce LIS is at least 2 sqrt n}
    For each $n\in\bN$, let $\bth^{(n)} = (\theta_1^{(n)} \ge \dots \ge \theta_n^{(n)})$ be a non-increasing sequence of positive weights, and let $\tau_n$ be a $\Luce{\bth^{(n)}}$ random permutation.
    Then, for any given $\varepsilon>0$:
    \[
    \prob{\LIS{\tau_n} \le (2-\varepsilon)\sqrt{n}} \cv{n\to\infty} 0 \,.
    \]
\end{corollary}

\begin{proof}
    By Lemma~\ref{lem: exponential representation}, we know that $\LIS{\tau_n}$ is distributed like $\LIS{E_1/\theta_1^{(n)}, \dots, E_n/\theta_n^{(n)}}$, where $E_1, \dots, E_n$ is a sequence of i.i.d.~$\Exp{1}$ random variables.
    The conclusion follows from \Cref{th: LIS uniform 2sqrt n} and Proposition~\ref{prop: superuniformity property} (with $c=0$ and $d=\infty$).
\end{proof}

\subsection{Proof of \Cref{th: LIS growth rate - general}}
\label{sec: proof growth rate}

\begin{lemma}\label{lem: upper bound on density gives upper bound on LIS}
    Let $Y_1, \dots, Y_m$ be a sequence of independent real random variables, where $Y_i$ has a density $g_i$ on $\bR$.
    Assume that there exists a density $g$ on $\bR$ and a constant $R\ge1$ such that for all $i\in[m]$ and $x\in\bR$, $g_i(x) \le R g(x)$.
    Then for any $\delta>0$:
    \[ 
    \prob{ \LIS{Y_1, \ldots, Y_m} \ge (1+\delta)e\sqrt{Rm} } \le \left(\frac{1}{1+\delta}\right)^{\sqrt{Rm}}
    \quad\text{and}\quad
    \expec{\LIS{Y_1, \ldots, Y_m}} \le 2e \sqrt{Rm} \,.
    \]
\end{lemma}

\begin{proof}
    For any $1\le \ell\le m$ and any subsequence $i_1 < \ldots < i_\ell$ of indices,
    \begin{equation*}
        \prob{Y_{i_1} < \ldots < Y_{i_\ell}}
        = \int_{y_1< \ldots< y_\ell} \prod_{1\le k\le \ell} \big( g_{i_k}(y_k) \mathrm dy_k \big)
        \le R^\ell \int_{y_1< \ldots< y_\ell} \prod_{1\le k\le \ell} \big( g(y_k) \mathrm dy_k \big)
        = \frac{R^\ell}{\ell!} \,.
    \end{equation*}
    Thus, a union bound and a Stirling bound lead to
    \begin{equation*}
        \prob{\LIS{Y_1, \ldots, Y_m} \ge \ell}
        \le \binom{m}{\ell} \frac{R^\ell}{\ell!}
        \le \frac{(mR)^\ell}{(\ell!)^2}
        \le \left( \frac{e^2Rm}{\ell^2} \right)^\ell
        \frac{1}{2\pi\ell} \,.
    \end{equation*}
    From there, the claimed probability bound is straightforward.
    Finally:
    \begin{equation*}
        \expec{\LIS{Y_1, \ldots, Y_m}}
        \le e\sqrt{Rm} + \sum_{\ell= \lceil e\sqrt{Rm}\rceil}^m \prob{\LIS{Y_1, \ldots, Y_m} \ge \ell}
        \le e\sqrt{Rm} + \sum_{\ell= \lceil e\sqrt{Rm}\rceil}^m \frac{1}{2\pi\ell} \,,
    \end{equation*}
    and the conclusion easily follows.
\end{proof}

\begin{figure}
    \centering
    \includegraphics[width=0.4\linewidth]{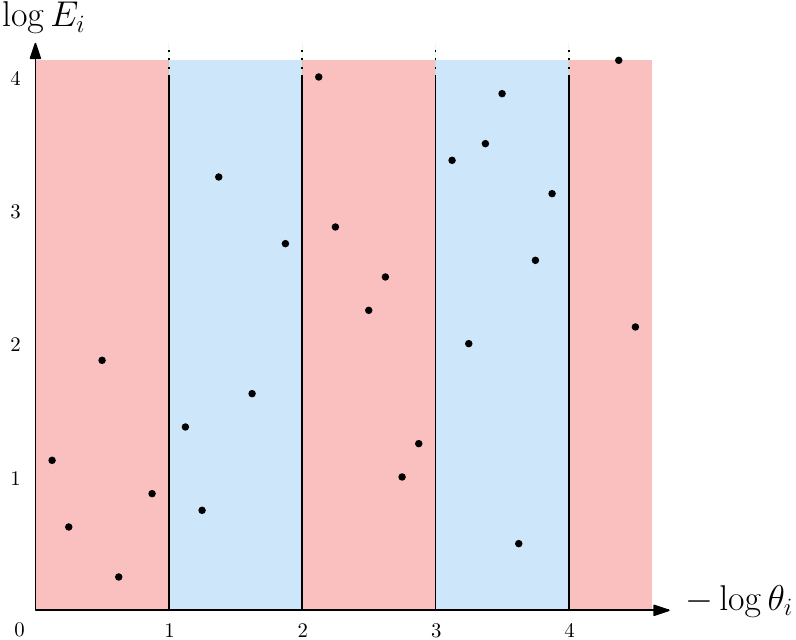}
    \qquad\qquad
    \includegraphics[width=0.4\linewidth]{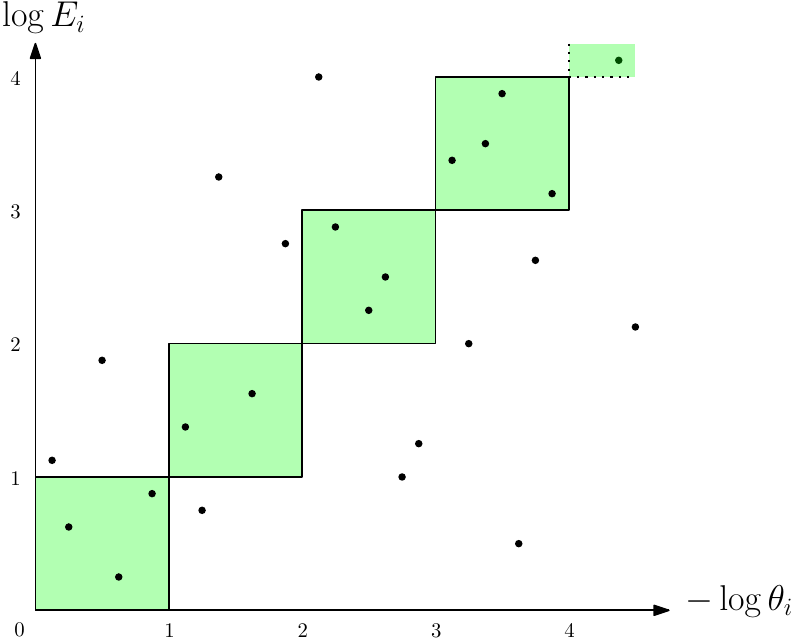}
    \caption{
    Illustration of the proof of \Cref{th: LIS growth rate - general}.
    The points $(\theta_i,E_i)$ are represented in log-scale.
    Left: 
    Lemma~\ref{lem: upper bound on density gives upper bound on LIS} provides an upper bound for the LIS on each band $[k,k+1) \times [0,\infty)$.
    Right:
    Proposition~\ref{prop: superuniformity property} provides a lower bound for the LIS on each square $[k,k+1)^2$.
    }
    \label{fig: growth rate bounds}
\end{figure}

\begin{proof}[Proof of \Cref{th: LIS growth rate - general}]
To ease notation, we may omit the dependence on $n$.
See \Cref{fig: growth rate bounds} for an illustration of this proof.\\
Let $E_1, \ldots, E_n$ be a sequence of independent random variables, where $E_i$ follows the $\Exp{\theta_i}$ distribution.
Then, by Lemma~\ref{lem: exponential representation}, $\LIS{\tau_n}$ is distributed like $\LIS{E_1, \ldots, E_n}$.

\medskip\noindent\textbf{Upper bound.}
Recall that for any $k\in\bZ$ and $i\in B_k$, $e^{-k-1} \le \theta_i \le e^{-k}$.
Therefore, the density of $E_i$ is bounded by $x\mapsto %e^{-k} e^{-e^{-k-1} x} = 
R g_k(x)$ with $R := e$ and $g_k$ the density of the $\Exp{e^{-k-1}}$ distribution.
Hence, by Lemma~\ref{lem: upper bound on density gives upper bound on LIS}:
\begin{equation*}
    \expec{\LIS{E_i ,\, i\in B_k}} \le 2e^{3/2} \sqrt{m_k} \,,
\end{equation*}
and we conclude by noting that 
\begin{equation*}
    \LIS{E_1, \ldots, E_n} \le \sum_{k\in\bZ} \LIS{E_i ,\, i\in B_k} \,.
\end{equation*}

\medskip\noindent\textbf{Lower bound.}
For $k\in\bZ$, set $J_k := [e^k, e^{k+1})$.
We can write:
\begin{equation}\label{eq: lower bound slicing}
    \LIS{E_1, \ldots, E_n} \ge \sum_{k\in\bZ} \LIS{\left\{ E_i : i\in B_k , E_i \in J_k \right\}} \,.
\end{equation}
Indeed, this follows from two observations.
Consider $k<k'$.
First, for any $i\in B_k$ and $j\in B_{k'}$, the monotonicity of the weights implies that $i<j$.
Second, if $E_i \in J_k$ and $E_j \in J_{k'}$, then $E_i < E_j$.
These two observations allow one to concatenate the increasing subsequences of each set $\left\{ E_i : i\in B_k , E_i \in J_k \right\}$ into an increasing subsequence of $\{ E_1, \ldots, E_n \}$.

Now, for each $k\in\bZ$, consider the random variables
\begin{equation*}
    R_k := \left\{ i\in B_k : E_i \in J_k \right\}
    \quad\text{and}\quad
    r_k := \#{R_k} \,.
\end{equation*}
For $i\in B_k$, conditionally on the event $i\in R_k$, the random variable $E_{i}$ follows the $\Exp{\theta_i}$ distribution conditioned to be in the interval $J_k$.
Therefore, Proposition~\ref{prop: superuniformity property} implies that
\begin{equation*}
    \expecond{\LIS{\left\{ E_i , i\in R_k \right\}}}{R_k}
    \ge \expecond{\LIS{\sigma_{r_k}}}{R_k}
    = \ell_{r_k} \,,
\end{equation*}
where $\sigma_{r_k}$ denotes a uniformly random permutation of size $r_k$, and we use the notation of \Cref{th: LIS uniform 2sqrt n}.
Using that $\ell_r \ge \sqrt{r}$ for all $r\in\bN_0$, we get that
\begin{equation}\label{eq: lower bound superuniform for conditioned exponentials}
    \expec{\LIS{\left\{ E_i , i\in R_k \right\}}}
    \ge \expec{\sqrt{r_k}} \,.
\end{equation}
Now, for $i\in B_k$, the probability that $E_i\in J_k$ is $e^{-\theta_i e^k} - e^{-e \theta_i e^k}$.
Using that $e^{-k-1} \le \theta_i \le e^{-k}$, we get that this probability is bounded below by a positive constant $p > 0$.
Therefore, letting $S_{m_k,p}$ denote a $\Binomial{m_k}{p}$ random variable:
\begin{equation}\label{eq: lower bound binomial}
    \expec{\sqrt{r_k}}
    \ge \expec{\sqrt{S_{m_k, p}}}
\end{equation}
In this context, it is easy to see that $\expec{\sqrt{S_{m,p}}} \sim \sqrt{mp}$ as $m\to\infty$, and thus, there exists a constant $c>0$ such that $\expec{\sqrt{S_{m,p}}} \ge c\sqrt{m}$ for all $m\in\bN_0$.
Combining this with \Cref{eq: lower bound slicing,eq: lower bound superuniform for conditioned exponentials,eq: lower bound binomial}, we finally get
\begin{equation*}
    \expec{\LIS{E_1, \ldots, E_n}}
    \ge c \sum_{k\in\bZ} \sqrt{m_k}
\end{equation*}
as desired.
\end{proof}

\subsection{A concentration inequality}
\label{sec: concentration}

In this short section, we explain why it is not restrictive to state \Cref{th: LIS growth rate - general} for the expectation of ${\LIS{\tau_n}}$, rather than for its tail probabilities.
Let us recall McDiarmid's inequality:

\begin{theorem}[\cite{McDiarmid_1989}]\label{th: McDiarmid general}
    Let $n\in\bN$, $Z_1, \dots, Z_n$ be independent random variables with values in a common space $\mathcal{Z}$, and $h : \mathcal{Z}^n \to \bR$ be a function satisfying the bounded differences property: 
    for all $i\in[n]$ and $z_1, \dots, z_n, w_i \in \mathcal{Z}$:
    \[
    \abs{ h(z_1, \dots, z_n) - h(z_1, \dots, w_i, \dots, z_n) } \le c
    \]
    for some constant $c>0$.
    Then, for any given $\lambda>0$:
    \[
    \prob{ \abs{\big. h(Z_1, \dots, Z_n) - \expec{h(Z_1, \dots, Z_n)}} > \lambda } \le 2\exp\left( \frac{-2\lambda^2}{nc^2} \right) \,.
    \]
\end{theorem}

This inequality is readily applicable to the study of LIS of random permutations; see, e.g., \cite{Dubach_2023} for the case of permuton samples.
The Luce distribution is no exception to this:

\begin{corollary}\label{cor: McDiarmid for Luce}
    Let $n\in\bN$ and $\tau_n$ be a Luce random permutation with positive weights $\theta_1, \dots, \theta_n$.
    Then, for all $\lambda>0$:
    \[
    \prob{ \abs{\big. \LIS{\tau_n} - \expec{\LIS{\tau_n}}} > \lambda } \le 2\exp\left( \frac{-2\lambda^2}{n} \right) \,.
    \]
\end{corollary}

\begin{proof}
    Using Lemma~\ref{lem: exponential representation}, we see that $\LIS{\tau_n}$ is distributed like $\LIS{E_1, \dots, E_n}$ where $E_i$, $i\in[n]$ are independent random variables with distributions $\Exp{\theta_i}$.
    McDiarmid's inequality (\Cref{th: McDiarmid general}) readily applies with the function $\mathrm{LIS}$, which changes by at most one when changing one of its coordinates.
\end{proof}

This allows us to deduce tail probabilities from estimates on $\expec{\LIS{\tau_n}}$.
In the context of \Cref{th: LIS growth rate - general}, one simple consequence is the following.

\begin{corollary}\label{cor: LIS is close to expectation if non-increasing weights}
    For each $n\in\bN$, let $\bth^{(n)} = (\theta_1 \ge \dots \ge \theta_n)$ be a non-increasing sequence of positive weights and let $\tau_n$ be a $\Luce{\bth^{(n)}}$ random permutation.
    %Assume that $\expec{\LIS{\tau_n}} = \Omega(\sqrt n)$ as $n\to\infty$.
    Then,
    \[
    \lim_{\lambda \to \infty} \liminf_{n\to\infty} \prob{ \LIS{\tau_n} \in \left[ \frac1\lambda \expec{\LIS{\tau_n}} , \lambda \expec{\LIS{\tau_n}} \right] } = 1 \,.
    \]
\end{corollary}

\begin{proof}
    Recall that, by Proposition~\ref{prop: superuniformity property} and \Cref{th: LIS uniform 2sqrt n}, $\expec{\LIS{\tau_n}} \ge \ell_n \ge \sqrt n$ holds for all $n\in\bN$.
    On the one hand, for $\lambda > 1$:
    \[
    \prob{\LIS{\tau_n} > \lambda \expec{\LIS{\tau_n}}}
    \le \prob{ \abs{\LIS{\tau_n} - \expec{\LIS{\tau_n}}} > (\lambda-1) \expec{\LIS{\tau_n}} }
    \le 2\exp\left( -2(\lambda-1)^2 \right)
    \]
    by Corollary~\ref{cor: McDiarmid for Luce}.
    This proves the upper tail decay\footnote{
    The upper tail decay can also be derived as an elementary consequence of Markov's inequality, but our proof provides an \emph{exponential} decay.
    }.
    Similarly, for $\lambda>2$:
    \[
    \prob{\LIS{\tau_n} < \frac1\lambda \expec{\LIS{\tau_n}}}
    \le \prob{ \abs{\big. \LIS{\tau_n} - \expec{\LIS{\tau_n}}} > \frac12 \expec{\LIS{\tau_n}} }
    \le 2\exp\left( -\frac{\expec{\LIS{\tau_n}}^2}{2n} \right) \,.
    \]
    Now, let us distinguish two cases.
    First, assume that $\expec{\LIS{\tau_n}} > \lambda \sqrt n$.
    In that case, the previous bound becomes:
    \[
    \prob{\LIS{\tau_n} < \frac1\lambda \expec{\LIS{\tau_n}}}
    \le 2\exp\left( -\lambda^2 / 2 \right) \,.
    \]
    Second, assume that $\expec{\LIS{\tau_n}} \le \lambda \sqrt n$.
    Then,
    \[
    \prob{\LIS{\tau_n} < \frac1\lambda \expec{\LIS{\tau_n}}}
    \le \prob{\LIS{\tau_n} < \sqrt n} \,,
    \]
    where the above goes to $0$ as $n\to\infty$ by Corollary~\ref{cor: Luce LIS is at least 2 sqrt n}.
    The above two bounds prove the lower tail decay, concluding the proof.
\end{proof}

\begin{remark}
    One could use Talagrand's inequality to get sharper tail bounds; see, e.g., \cite{Tal95} or \cite[Theorem~2.29]{janson_graphs}.
    We leave such details to the reader.
\end{remark}

\section{Computing the constant in the $\sqrt n$ regime}
\label{sec: the constant}

\subsection{The continuous Luce model}
\label{sec: two models}

Let us discuss here what we may call the ``continuous'' Luce model, in opposition with the ``discrete'' Luce model.

Recall (\Cref{th: Luce limit permuton}) that if $\bth^{(n)} = (\theta_1^{(n)}, \ldots, \theta_n^{(n)})$ are weights that admit a finite, positive limit $f$ on $[0,1)$ in the right sense, then Luce permutations with weights $\bth^{(n)}$ admit a deterministic permuton limit $\mu_f$.
Further, we can reformulate \cref{eq: Luce pattern densities} as follows (see, also, the remark after \cite[Proposition~2.2]{Borga_Chatterjee_Diaconis_2025}).

\begin{lemma}\label{lem: permuton sample is Luce with random weights}
    Let $\mu_f$ be a Luce permuton as in \Cref{th: Luce limit permuton}.
    Then, for all $n\in\bN$, the random permutation $\sigma_{n,\mu_f}$ sampled from $\mu_f$ follows the Luce distribution with weights $f(U_{(1)}), \ldots, f(U_{(n)})$, where $U_{(1)} < \ldots < U_{(n)}$ denotes the increasing reordering of i.i.d.\ uniform r.v.'s on $[0,1]$.
    That is:
    \begin{equation*}\label{eq: permuton sample is Luce with random weights}
        \prob{\sigma_{n,\mu_f} = \sigma}
        = \expec{ \probcond{\tau_n^{\bU} = \sigma}{U_1, \ldots, U_n} } \,,
    \end{equation*}
    for any $\sigma\in\fS_n$,
    where we let $\tau_n^{\bU}$ denote a Luce random permutation with weights $f(U_{(1)}), \ldots, f(U_{(n)})$.
\end{lemma}

Therefore, we may say that $\sigma_{n,\mu_f}$ follows the \emph{continuous Luce distribution with parameter $f$}.
In comparison, we call \emph{discrete Luce distribution with parameter $f$} the Luce distribution with weights $\bth^{(n)}(f) := \big( f\big(\frac{i-1}{n}\big) \big)_{1\le i\le n}$.
The latter choice of weights is quite natural for the Luce model (they include, e.g., the famed Sukhatme weights $\theta_i = (n-i+1)/n$).

\begin{remark}
    We choose here to stick with the ``canonical'' choice of weights $\theta_i^{(n)} = f\big( (i-1)/n \big)$ for convenience, but it seems likely that our results could be extended to allow weights of the form $f\big( x_i^{(n)} \big)$ for ``sufficiently regular'' subdivisions $(x_1^{(n)} < \ldots< x_n^{(n)})$.
\end{remark}

The main goal of this section is to argue that the discrete and continuous Luce models display similar behaviors.
In particular, we show in \Cref{th: LLN with variational constant} that their LIS admit the same limit constant in the $\sqrt n$ regime, allowing us to transfer known results about the LIS of permuton samples (see \cite{Deuschel_Zeitouni_1995}) to the LIS of (discrete) Luce permutations.
As a warm-up, we start by showing that both models share the same growth rate.

\begin{lemma}\label{lem: discrete and continuous models share same groth rate}
    Let $f$ be a positive, finite, non-increasing function on $[0,1)$ {normalized so that $f(0)=1$}, and let $\mu_f$ be the associated permuton of \Cref{th: Luce limit permuton}.
    For each $n\in\bN$, let $\sigma_{n,\mu_f}$ be a $\mu_f$-sample.
    Then:
    \begin{equation*}
        \expec{\LIS{\sigma_{n,\mu_f}}}
        \asymp \psi_f(n) \,,
    \end{equation*}
    where $n\mapsto \psi_f(n)$ is the function of Proposition~\ref{prop: growth rate with phi}.
\end{lemma}

\begin{proof}
    Let $\bU := (U_1, \ldots, U_n)$ be a sequence of i.i.d.\ uniform random variables on $[0,1]$, and let $U_{(1)} < \ldots < U_{(n)}$ denote their (a.s.\ unique) increasing reordering.
    Conditionally given these, let $\tau_n^\bU$ denote a Luce random permutation with weights $\big( f\big( U_{(i)} \big) \big)_{1\le i\le n}$, so that the unconditional law of $\tau_n^\bU$ is $\sample_n(\mu_f)$ by Lemma~\ref{lem: permuton sample is Luce with random weights}.

    We can apply \Cref{th: LIS growth rate - general} to $\tau_n^\bU$, conditionally given $\bU$, and then take expectation to get that
    \[
        \expec{ \expecond{\LIS{\tau_n^\bU}}{\bU} }
        \asymp \expec{ \sum_{k\in\bN_0} \sqrt{m_k^{(n)}(\bU)} } \,,
    \]
    where 
    \[
        m_k^{(n)}(\bU) := \card{ i\in[n]: k \le -\log f(U_{(i)}) < k+1 } \,.
    \]
    With the notation of Proposition~\ref{prop: growth rate with phi}, we see that the unconditional law of $m_k^{(n)}(\bU)$ is the binomial distribution $\Binomial{n}{p_k}$.
    The conclusion follows from Lemma~\ref{lem: sqrt of binomial}.
\end{proof}

\begin{remark}
    If $f$ is $\log$-concave, then this coincides with the growth rate of $\LIS{\tau_n}$ where $\tau_n$ follows the $\Luce{\bth^{(n)}(f)}$ distribution (i.e., the \enquote{discrete} model).
\end{remark}

\begin{example}
    Despite our philosophy here being that Luce random permutations $\tau_n$ should behave ``almost like'' samples $\sigma_{n,\mu_f}$ of their permuton limit, we should stress that this crucially requires some regularity on the sequence of weights.
    As an example, choose $f: x\in[0,1] \mapsto 1-x$, two sequences $(h_n), (w_n)$ such that $1\le h_n\le n$ and $w_n > 0$, and set
    \[
    \theta_i^{(n)} := \begin{cases}
        w_n^{h_n-i+1} &\text{if } 1\le i\le h_n;\\
        \frac{n-i+1}{n} &\text{if } h_n< i\le n.
    \end{cases}
    \]
    If $h_n = o(n)$, then \Cref{th: Luce limit permuton} applies and Luce random permutations $\tau_n$ with weights $\big( \theta_i^{(n)} \big)$ converge in probability towards the permuton $\mu_f$.
    However, this alone tells us very little on the behavior of $\LIS{\tau_n}$.
    Indeed, if $E_1, \ldots, E_n$ denotes a sequence of independent exponential r.v.'s with respective parameters $\theta_i^{(n)}$, then for $i< h_n$:
    \[
        \prob{E_i > E_{i+1}}
        = \frac{\theta_{i+1}^{(n)}}{\theta_i^{(n)} + \theta_{i+1}^{(n)}}
        = \frac{1}{w_n +1} \,,
    \]
    and thus:
    \[
        \prob{E_1 < \ldots < E_{h_n}}
        \ge 1 - \frac{h_n-1}{w_n+1} \,.
    \]
    If we choose our sequences to satisfy $h_n = o(w_n)$, then this goes to $1$.
    Therefore, by Lemma~\ref{lem: exponential representation}, $\LIS{\tau_n}$ is at least $h_n$ with high probability.
    If we choose $h_n := \lfloor n^\beta \rfloor$ for some $\beta \in (1/2, 1)$ and $w_n := n$, then $\LIS{\tau_n}$ is at least $\lfloor n^\beta \rfloor$ with high probability, contrasting with the $\sqrt n$ growth rate for Sukhatme weights (cf Example~\ref{ex: Sukhatme}).
\end{example}

To conclude this section, let us formulate an equivalent description of the comparison between the discrete model and the continuous model.
Although we will not formally need it, we believe it helps better understand the current situation.
Beforehand, we need to introduce some notation.

Let $\mu$ be a permuton with density $\rho$, and assume that $\rho(x,\cdot)$ is a probability density for all $x\in[0,1)$.
This is, e.g., the case for the Luce density \eqref{eq: density of the Luce permuton}, as $\rho_f(x,\cdot)$ is the density of $F(E/f(x))$.
Consider $0\le x_1\le \ldots\le x_n< 1$. 
Then, we let $\sigma_{x_1, \ldots, x_n,\rho}$ denote the random permutation $\Perm{Y_1, \ldots, Y_n}$ induced by a sequence of independent random variables $(Y_i)_{1\le i\le n}$ where $Y_i$ is distributed under the density $\rho(x_i,\cdot)$.
We let $\sample_{x_1, \dots, x_n}(\rho)$ denote the law of $\sigma_{x_1, \ldots, x_n,\rho}$.
With this formalism, observe that the distribution $\sample_n(\mu)$ is the average of $\sample_{U_{(1)}, \dots, U_{(n)}}(\rho)$ where $U_{(1)}< \dots< U_{(n)}$ is the reordering of i.i.d.~uniform coordinates.

\begin{lemma}\label{lem: discrete Luce is permuton sample with fixed x-coordinates}
    Let $f$ be a positive, finite, measurable function on $[0,1)$.
    Let $\mu_f$ be the permuton with density $\rho_f$ given by \cref{eq: density of the Luce permuton}.
    Then, for any $0\le x_1\le \ldots\le x_n< 1$,
    \begin{equation*}
        \sample_{x_1, \ldots, x_n}(\rho_f) = \Luce{\big. f(x_1), \ldots, f(x_n)} \,.
    \end{equation*}
\end{lemma}

\begin{proof}
    Recall from \Cref{th: Luce limit permuton} that $\rho_f(x,\cdot)$ is the density of the random variable $F(E/f(x))$ where $E\sim\Exp{1}$.
    Note that $E/f(x)$ is an $\Exp{f(x)}$ random variable, and that $F$ is an increasing map.
    Therefore, $\sample_{x_1, \dots, x_n}(\rho_f)$ coincides with the law of the random permutations $\Perm{E_1, \dots, E_n}$ where $E_i \sim \Exp{f(x_i)}$ are independent.
    The conclusion follows by Lemma~\ref{lem: exponential representation}.
\end{proof}

Let us briefly summarize our two comparison results.
On the one hand, Lemma~\ref{lem: permuton sample is Luce with random weights} states the continuous Luce model is like a discrete Luce model \emph{with random weights}.
On the other hand, Lemma~\ref{lem: discrete Luce is permuton sample with fixed x-coordinates} states that the discrete Luce model is like a continuous Luce model (i.e., a permuton sample) \emph{with specified x-coordinates}.
% In other words, while the continuous model $\sample_n(\mu_f)$ is constructed via an \enquote{unrestricted} $\mu$-i.i.d.\ point set, the discrete model $\Luce{\bth^{(n)}(f)}$ is constructed via a $\mu$-i.i.d.\ point set \emph{with specified x-coordinates}.
Both points of view are particularly well-adapted to the use of coupling techniques, although we found it slightly easier to use the first one.
This is illustrated in the next section.

\subsection{Coupling the discrete and the continuous models}
\label{sec: coupling}

In view of Lemma~\ref{lem: permuton sample is Luce with random weights}, one should expect that the discrete model $\Luce{\bth^{(n)}(f)}$ can be coupled with the continuous model $\sample_n(\mu_f)$ in a rather efficient way.
We make this formal in the next result.

\begin{proposition}\label{prop: coupling discrete and continuous}
    Let $f$ be a positive, finite, $\cC^1$ function on $[0,1)$, 
    normalized so that $\sup f \le 1$,
    and let $\mu_f$ be the associated permuton of \Cref{th: Luce limit permuton}.
    Fix $\varepsilon\in(0,1)$.
    Then, there exists a coupling of $\tau_n \sim \Luce{\bth^{(n)}(f)}$ and $\sigma_{n,\mu_f} \sim \sample_n(\mu_f)$ together with a random subset $\cI_\unmatched \subseteq [(1-\varepsilon)n]$ such that:
    \begin{enumerate}[label=(\roman*)]
        \item $ 
        \Perm{ \left\{ \tau_n(i) : i\in[(1-\varepsilon)n] \setminus \cI_\unmatched \right\} }
        = \Perm{ \left\{ \sigma_{n,\mu_f}(i) : i\in[(1-\varepsilon)n] \setminus \cI_\unmatched \right\} }
        $;
        \item $
        \expec{\#{\cI_\unmatched}} = \cO\big( \sqrt n (\log n)^{3/2} \big)
        $;
        \item $
        \max\left\{\expec{ \LIS{ \left\{ \tau_n(i) : i\in\cI_\unmatched \right\} } },
        \expec{ \LIS{ \left\{ \sigma_{n,\mu_f}(i) : i\in\cI_\unmatched \right\} } }\right\}
        = \cO\big( n^{1/4} (\log n)^{3/4}\big)
        $.
    \end{enumerate}
\end{proposition}

\begin{proof}
    For $i\in[n]$, we write $x_i := \frac{i-1}{n}$.
    We let $(E_i)_{1\le i\le n}$ denote a sequence of independent random variables where $E_i \sim \Exp{f(x_i)}$, so that $\tau_n := \Perm{E_1, \ldots, E_n}$ follows the $\Luce{\bth^{(n)}(f)}$ distribution by Lemma~\ref{lem: exponential representation}.
    
    Let $\bU := \{ U_i : i\in[n] \}$ denote a sequence of i.i.d.\ uniform random variables on $[0,1]$, and let $U_{(1)} < \ldots < U_{(n)}$ denote their increasing reordering.
    Conditionally given $\bU$, let $(T_i)_{1\le i\le n}$ denote a sequence of independent random variables where $T_i \sim \Exp{f\big( U_{(i)} \big)}$, so that the unconditional law of $\sigma_n := \Perm{T_1, \ldots, T_n}$ is $\sample_n(\mu_f)$ by Lemma~\ref{lem: permuton sample is Luce with random weights}.
    Conditionally given $\bU$, we shall use Lemma~\ref{lem: exponential coupling} to construct an appropriate coupling of $(T_i)_{1\le i\le n}$ with $(E_i)_{1\le i\le n}$.
    
    Set $t_n := \sqrt{\log n}$.
    Thanks to the DKW inequality (Corollary~\ref{cor: DKW inequality for uniform variables}), we have that
    \begin{equation}
    \prob{ A_n } \ge 1 - 2e^{-2t_n^2} = 1 - \frac{2}{n^2}
    \quad\text{where}\quad
    A_n := \left\{ \sup_{1\le i\le n} \abs{ U_{(i)} - x_i } \le \frac{t_n}{\sqrt n} \right\} \,.
    \end{equation}
    Observe that for $n$ large enough, under the event $A_n$, if $i\in[n]$ is such that $x_i \in [0,1-\varepsilon]$ then $U_{(i)} \in [0,1-\frac12 \varepsilon]$.
    Let 
    \[ \eta := \min\left\{1/4,\min_{[0,1-\frac{1}{2}\varepsilon]} f \right\} > 0
    \quad\text{and}\quad
    \delta_n := \frac{t_n}{\sqrt n} \max\left\{1,\max_{[0,1-\frac{\varepsilon}{2}]} \abs{f'}\right\} 
    = o(1) \,,
    \]
    and write $\theta_i := f(x_i)$, $\theta_i' := f\big( U_{(i)} \big)$ for $i\in[n]$,
    so that under the event $A_n$, for all $i\in[(1-\varepsilon)n]$, we have $\theta_i, \theta_i' \in [\eta, 1]$ and $\abs{\theta_i - \theta_i'} \le \delta_n$.
    Conditionally given $\bU$ and under the event $A_n$, we can apply Lemma~\ref{lem: exponential coupling} independently to each $i\in[(1-\varepsilon)n]$ to get a coupling of $(E_i)_{1\le i\le n}$ with $(T_i)_{1\le i\le n}$ such that, for all $i\in[(1-\varepsilon)n]$:
    \begin{itemize}
        \item $
        \prob{E_i \ne T_i} \le \frac{2\delta_n}{\eta} \big( 1 + \log \frac{1}{2\delta_n} \big) =: q_n 
        $;
        \item the law of $E_i$ conditionally on the event $\{ E_i \ne T_i \}$ has a density bounded by $Rg_n$, where $R\in(0,\infty)$ is an absolute constant and $g_n$ is a probability density on $[0,\infty)$ that depends only on $n, \varepsilon, f$.
    \end{itemize}
    The random variables $E_i$ and $T_i$ for $i > (1-\varepsilon)n$ are coupled in an arbitrary way, and will not intervene whatsoever.
    Under the complementary event $\neg A_n$, we couple $(E_i)_{1\le i\le n}$ with $(T_i)_{1\le i\le n}$ in an arbitrary way.
    
    We set $\cI_\unmatched := \left\{ i\in[(1-\varepsilon)n] : E_i \ne T_i \right\}$, so that the first claimed item is clear by construction.
    Moreover, under the event $A_n$, the size of $\cI_\unmatched$ is stochastically bounded by the $\Binomial{n}{q_n}$ distribution.
    Under the complementary event $\neg A_n$, we simply bound the size by $n$.
    Therefore:
    \[
        \expec{\#{\cI_\unmatched}}
        \le nq_n + 2n^{-1}
        = \cO\big( \sqrt n (\log n)^{3/2} \big) \,.
    \]
    This proves the second item.
    Finally, we can apply Lemma~\ref{lem: upper bound on density gives upper bound on LIS} to the sequence $(E_i, i\in\cI_\unmatched)$ to get
    \[
        \expec{\expecond{ \LIS{ E_i : i\in\cI_\unmatched} }{\bU, A_n, \cI_\unmatched}}
        \le \expec{2e\sqrt{R \#{\cI_\unmatched}}}
        \le \cO\big( n^{1/4} (\log n)^{3/4}\big) \,,
    \]
    where we used Jensen's inequality.
    It suffices to bound the LIS by $n$ under the event $\neg A_n$ to prove our third item, for $\tau_n$.
    The proof works the same for $\sigma_n$.
\end{proof}

\subsection{Retrieving the limit constant from the continuous model}
\label{sec: proof the constant}

This section is devoted to our second main result, \Cref{th: LLN with variational constant}.
We start with a few preliminary results that will ease its proof.

\begin{lemma}\label{lem: induced Luce}
    Let $(\theta_i)_{1\le i\le n}$ be a sequence of positive weights, and let $\tau_n$ denote a Luce random permutation with weights $\theta_1, \ldots, \theta_n$.
    Then, for any given $k\in[n]$ and $1\le i_1< \ldots< i_k\le n$, the induced subpermutation $\Perm{\tau_n(i_1), \ldots, \tau_n(i_k)}$ follows the Luce distribution with weights $\theta_{i_1}, \ldots, \theta_{i_k}$.
\end{lemma}

\begin{proof}
    This is immediate by Lemma~\ref{lem: exponential representation}.
\end{proof}

\begin{lemma}\label{lem: LIS of high exponentials}
    Fix $A > 0$, $m\in\bN$, and $b>a>0$.
    Let $\theta_1, \ldots, \theta_m \in [a,b]$, and let $(E_i)_{1\le i\le m}$ denote a sequence of independent exponential random variables with parameters $(\theta_i)_{1\le i\le m}$, each conditioned to be larger than $A$.
    Then:
    \[
    \prob{\LIS{E_1, \ldots, E_m} \ge 2e \sqrt{\frac{b}{a}} \sqrt{m}} \le \frac{1}{2^{\sqrt{m}}}
    \quad\text{and}\quad
    \expec{\LIS{E_1, \ldots, E_m}} \le 2e \sqrt{\frac{b}{a}} \sqrt{m} \,.
    \]
\end{lemma}

\begin{proof}
    This is immediate from Lemma~\ref{lem: upper bound on density gives upper bound on LIS}, and the observation that $(E_i)_{1\le i\le m}$ is distributed like $(A+T_i)_{1\le i\le m}$ where $(T_i)_{1\le i\le m}$ denotes a sequence of independent exponential random variables with parameters $(\theta_i)_{1\le i\le m}$.
\end{proof}

\begin{corollary}\label{cor: bound on LIS north-west band}
    Let $f$ be a positive, finite, continuous function on $[0,1)$, and let $\mu_f$ denote the associated permuton of \Cref{th: Luce limit permuton}.
    For each $n\in\bN$, let $\sigma_n$ denote a $\mu_f$-sample.
    %$\tau_n$ denote a Luce random permutation with weights $\bth^{(n)}(f) := \big( f\big( \frac{i-1}{n} \big) \big)_{1\le i\le n}$.
    Then, for any given $0< \alpha, \beta< 1$:
    \begin{equation}\label{eq: bound on LIS north-west band}
        \limsup_{n\to\infty} \frac{1}{\sqrt{n}} {\LIS{ \left\{ \sigma_n(i) : i\le \alpha n , \sigma_n(i) \ge \beta n \right\} }}
        \le \eta_{\alpha,\beta} \,,
    \end{equation}
    in probability and in expectation,
    where $(\eta_{\alpha,\beta})_{0< \alpha, \beta< 1}$ is a family of finite constants such that, for any given $\alpha\in(0,1)$,
    \[
        \eta_{\alpha,\beta} \cv{\beta\to1} 0 \,.
    \]
    %The same holds when replacing $\tau_n$ with $\sigma_{n,\mu_f} \sim \sample_n(\mu)$.
\end{corollary}

\begin{proof}
    Let $(U_i)_{1\le i\le n}$ be a sequence of i.i.d.\ uniform random variables, conditioned to be ordered as $U_{1}< \ldots< U_{n}$.
    Let $(E_i)_{1\le i\le n}$ be a sequence of i.i.d.\ standard exponential r.v.'s, and for each $i\in[n]$, let $V_i := F\big( E_i/f(U_{i}) \big)$ with the notation of \Cref{th: Luce limit permuton}.
    By that same theorem, $\cP := \left\{ (U_{i},V_i) : i\in[n] \right\}$ is distributed like a $\mu_f$-i.i.d.\ point set, and $\sigma_n := \Perm{\cP} =  \Perm{V_1, \ldots, V_n}$ is distributed like a $\mu_f$-sample.
    For $0<\alpha,\beta<1$, we let 
    \[
    I_\alpha := \max \left\{ i\in[n] : U_{i} \le \alpha \right\} 
    \quad\text{and}\quad
    J_\beta := \max \left\{ \sigma_n(i) : i\in[n] , V_{i} < \beta \right\} \,.
    \]
    Since $\mu_f$ is a permuton, we know that $\{ U_i : i\in[n] \}$ and $\{ V_i : i\in[n] \}$ are each distributed like a point set of i.i.d.\ uniform points on $[0,1]$.
    Therefore, by the Law of Large Numbers (or the DKW inequality, cf Corollary~\ref{cor: DKW inequality for uniform variables}):
    \begin{equation}\label{eq: LLN for indices}
    \frac1n I_\alpha \cv{n\to\infty} \alpha
    \quad\text{and}\quad
    \frac1n J_\beta \cv{n\to\infty} \beta \,,
    \end{equation}
    almost surely.
    Let us write here
    \[
        \sigma_n^{\langle I_\alpha, J_\beta \rangle} := \Perm{ \sigma_n(i) 
        \,:\,
        i \le I_\alpha , \sigma_n(i) > J_\beta } \,,
    \]
    and note that it coincides with
    \[
        \sigma_n^{\langle I_\alpha, J_\beta \rangle} = \Perm{\left\{\big. (U_i, V_i) 
        \,:\,
        i\in[n], U_i \le \alpha, V_i \ge \beta \right\}} 
        = \Perm{\cP \cap B_{\alpha,\beta}} \,,
    \]
    with the notation of \Cref{sec: notation} and $B_{\alpha,\beta} := [0,\alpha] \times [\beta,1]$.
    In words, $\sigma_n^{\langle I_\alpha, J_\beta \rangle}$ is constructed with the same $\mu_f$-i.i.d.\ point set $\cP$ as $\sigma_n$, but restricting it to the north-west band $B_{\alpha,\beta}$.
    By monotony and thanks to \eqref{eq: LLN for indices}, it suffices to prove \eqref{eq: bound on LIS north-west band} for the LIS of $\sigma_n^{\langle I_\alpha, J_\beta \rangle}$.
    First, note that the (random) size $N_{\alpha,\beta}$ of this permutation follows a binomial distribution with parameters $n$ and $\mu_f(B_{\alpha,\beta})$.
    Then, conditionally given the set 
    \[
    \cI := \left\{ i\in[n] : (U_i, V_i) \in B_{\alpha,\beta} \right\} \,,
    \]
    and the sequence $\bU := (U_i)_{1\le i\le n}$,
    the family $\big( F^{-1}(V_i) \big)_{i\in \cI}$ consists of independent exponential r.v.'s with parameters $\big( f(U_i) \big)_{i\in\cI}$, each conditioned to be larger than $F^{-1}(\beta)$.
    Therefore, by Lemma~\ref{lem: LIS of high exponentials}, 
    \[
    \probcond{\LIS{ \left\{ V_i : i\in\cI \right\}} \ge 2e \sqrt{\lambda_\alpha N_{\alpha,\beta}}}{\bU,\cI}
    \le \frac{1}{2^{\sqrt{N_{\alpha,\beta}}}}   
    \quad\text{and}\quad
    \expecond{\LIS{ \left\{ V_i : i\in\cI \right\}}}{\bU,\cI}
    \le 2e \sqrt{\lambda_\alpha N_{\alpha,\beta}} 
    \,,
    \]
    where $\lambda_\alpha := \frac{ \max_{[0,\alpha]} f }{ \min_{[0,\alpha]} f } \in (0,\infty)$.
    On the one hand, 
    observe that $\frac1n N_{\alpha,\beta}\to\mu_f(B_{\alpha,\beta}) > 0$ almost surely as $n\to\infty$, for any given $0<\alpha,\beta<1$.
    On the other hand,
    by Jensen's inequality, we can bound the unconditional expectation by
    \[
    \expec{\LIS{ \left\{ V_i : i\in\cI \right\}}}
    \le 2e \sqrt{\lambda_\alpha \mu_f(B_{\alpha,\beta}) n} 
    \,.
    \]
    Setting $\eta_{\alpha,\beta} := 2e\sqrt{\lambda_\alpha \mu_f(B_{\alpha,\beta})}$, the conclusion readily follows.
\end{proof}

\begin{remark}
    Corollary~\ref{cor: bound on LIS north-west band} shows that the LIS of Luce permuton samples (and permutations) does not concentrate on the north-west band $B_{\alpha,\beta}$ when $\beta\to1$ (see, e.g., the blue band in \Cref{fig: LLN zones}).
    Similarly, we could prove that the LIS does not concentrate either on the south-east band $B_{\alpha,\beta}^* := [\beta,1]\times[0,\alpha]$, under the same assumption.
    In the proof of \Cref{th: LLN with variational constant}, we will show that the LIS does not concentrate either on the north-east square $S_\alpha := [\alpha,1]^2$ as $\alpha\to1$ (see, e.g., the red band in \Cref{fig: LLN zones}), but this crucially requires the assumption that the parameter $f$ gives rise to a $\sqrt n$ growth rate.
    For functions giving rise to a higher growth rate, the LIS will always concentrate on $S_\alpha$, making its study more difficult.
\end{remark}

\begin{proposition}\label{prop: LLN for restricted permuton samples}
    Let $\mu$ be a permuton with density $\rho$ which is $\cC^1$, positive and bounded on $[0,\alpha]^2$ for all $\alpha\in(0,1)$.
    For each $n\in\bN$, let $\sigma_{n}$ denote a $\mu$-sample, and for any $\alpha,\beta\in(0,1]$, write 
    \[
    \sigma_{n}^{(\alpha,\beta)} := \Perm{\left\{ \sigma_{n}(i) : i\le\alpha n , \sigma_n(i) \le \beta n \right\}} \,.
    \]
    Then, there exists a family of finite constants $(K_{\alpha,\beta})_{0<\alpha,\beta<1}$ such that, for each $0< \alpha, \beta< 1$,
    \[
        \lim_{n\to\infty} \frac{1}{\sqrt n} \LIS{\sigma_n^{(\alpha,\beta)}}
        = 2 K_{\alpha,\beta} \,,
    \]
    in probability and in expectation,
    and
    \[
        K_{\alpha,\beta} \cv{(\alpha,\beta)\to(1,1)}
        K
        := \sup_{\phi \nearrow} \int_0^1 \sqrt{ \phi'(t) \rho(t,\phi(t)) } \mathrm{d}t \,,
    \]
    where the supremum ranges over all non-decreasing, right-continuous curves $\phi:[0,1)\to[0,1)$.
\end{proposition}

\begin{proof}
    Let $\alpha,\beta\in(0,1)$ and write $R_{\alpha,\beta} := [0,\alpha] \times [0,\beta]$.
    Let $\rho_{\alpha,\beta}$ denote the probability density induced by $\rho$ on $R_{\alpha,\beta}$, that is,
    \[
    \rho_{\alpha,\beta} = \frac{1}{\mu(R_{\alpha,\beta})} \rho \One{R_{\alpha,\beta}} \,.
    \]
    Write $\sigma_n = \Perm{\cP}$, where $\cP$ is a $\mu$-i.i.d.\ set of $n$ points in $[0,1]^2$.
    Let $\tilde \sigma_n^{(\alpha,\beta)} := \Perm{\cP \cap R_{\alpha,\beta}}$ be the permutation induced by the points in $R_{\alpha,\beta}$.
    Then, $\tilde \sigma_n^{(\alpha,\beta)}$ differs from $\sigma_n^{(\alpha,\beta)}$, but as in the proof of Corollary~\ref{cor: bound on LIS north-west band}, we can write $\tilde \sigma_n^{(\alpha,\beta)} = \sigma_n^{(I_\alpha/n,J_\beta/n)}$ for random variables $I_\alpha, J_\beta$ that are a.s.\ asymptotic to $\alpha n , \beta n$.
    Thus, by monotony, it suffices to prove Proposition~\ref{prop: LLN for restricted permuton samples} for $\tilde \sigma_n^{(\alpha,\beta)}$.
    Now, observe that the random permutation $\tilde \sigma_n^{(\alpha,\beta)}$ has a random size $N_{\alpha,\beta} \sim \Binomial{n}{\mu(R_{\alpha,\beta})}$, and conditionally given $N_{\alpha,\beta}$, it is distributed like $\Perm{\tilde\cP}$ where $\tilde\cP$ consists of $N_{\alpha,\beta}$ i.i.d.\ random points in $R_{\alpha,\beta}$ distributed under the density $\rho_{\alpha,\beta}$.
    Now, let $\tilde\rho_{\alpha,\beta}$ denote the probability density on $[0,1]^2$ obtained by ``stretching'' $\rho_{\alpha,\beta}$:
    \begin{equation*}
        \text{for all }(x,y)\in[0,1]^2,\quad
        \tilde\rho_{\alpha,\beta}(x,y) := 
        \alpha \beta \rho_{\alpha,\beta}(\alpha x , \beta y)
        = \frac{\alpha \beta}{\mu(R_{\alpha,\beta})} 
        \rho( \alpha x , \beta y) \,.
    \end{equation*}
    Note that the $\mathrm{Perm}$ map is invariant under linear transformations.
    Therefore, conditionally given $N_{\alpha,\beta}$, the random permutation $\tilde \sigma_n^{(\alpha,\beta)}$ is distributed like\footnote{
    We avoid the terminology ``permuton sample'' here, as $\tilde\rho_{\alpha,\beta}$ is not a \emph{permuton} density.
    With the terminology of \cite{Dubach_2023}, it is a \emph{prepermuton}.
    } 
    $\Perm{\tilde \cP}$ where $\tilde \cP$ consists of $N_{\alpha,\beta}$ i.i.d.\ random points in $[0,1]^2$ distributed under the density $\tilde\rho_{\alpha,\beta}$.
    Since $\tilde\rho_{\alpha,\beta}$ is $\cC^1$ and bounded below on $[0,1]^2$, we can apply \cite[Theorem~2]{Deuschel_Zeitouni_1995}:
    \begin{equation*}
        \frac{1}{\sqrt{N_{\alpha,\beta}}} \LIS{ \tilde \sigma_n^{(\alpha,\beta)} } 
        \cv{n\to\infty} 
        2 \sup_{\phi \nearrow} \int_0^1 \sqrt{ \phi'(t) \tilde\rho_{\alpha,\beta}(t,\phi(t)) } \mathrm{d}t
    \end{equation*}
    in probability and in expectation, and further:
    \begin{equation*}
        \frac{1}{\sqrt{n}} \LIS{ \tilde \sigma_n^{(\alpha,\beta)} } 
        \cv{n\to\infty} 
        2 \sqrt{\mu(R_{\alpha,\beta})} \sup_{\phi \nearrow} \int_0^1 \sqrt{ \phi'(t) \tilde\rho_{\alpha,\beta}(t,\phi(t)) } \mathrm{d}t
        =: 2K_{\alpha,\beta}
    \end{equation*}
    in probability and in expectation.
    It remains only to show that $K_{\alpha,\beta} \to K$ as $(\alpha,\beta) \to (1,1)$.
    First, note that $\mu(R_{\alpha,\beta}) \to \mu\big( [0,1]^2 \big) = 1$ as $(\alpha,\beta)\to(1,1)$, thanks to the assumption that $\mu$ is a permuton.
    Then, for $\alpha,\beta \in (0,1]$, let $\cC^\nearrow(\alpha,\beta)$ denote the set of non-decreasing right-continuous curves $\phi: [0,\alpha) \to [0,\beta)$.
    An elementary change of variables allows us to rewrite
    \begin{align*}
        \sup_{\phi \in \cC^\nearrow(1,1)} \int_0^1 \sqrt{ \phi'(t) \tilde\rho_{\alpha,\beta}(t,\phi(t)) } \mathrm{d}t
        = \frac{1}{\sqrt{\mu(R_{\alpha,\beta})}} 
        \sup_{\phi \in \cC^\nearrow(\alpha,\beta)} \int_0^{\alpha} \sqrt{ \phi'(t) \rho(t, \phi(t)) } \mathrm{d}t \,,
    \end{align*}
    and thus, it remains to show that
    \begin{align}\label{eq: convergence of DZ constant after rewriting}
        \sup_{\phi \in \cC^\nearrow(\alpha,\beta)} \int_0^{\alpha} \sqrt{ \phi'(t) \rho(t, \phi(t)) } \mathrm{d}t
        \cv{(\alpha,\beta)\to(1,1)}
        \sup_{\phi \in \cC^\nearrow(1,1)} \int_0^{1} \sqrt{ \phi'(t) \rho(t, \phi(t)) } \mathrm{d}t \,.
    \end{align}
    It is clear that the left-hand side is less than or equal to the right-hand side.
    Conversely, consider $\phi\in \cC^\nearrow(1,1)$.
    Set
    \begin{equation*}
        t_{\alpha,\beta} := \sup\{ t\in[0,\alpha) : \phi(t)< \beta^2 \} \,.
    \end{equation*}
    Since $t_{\alpha,\beta} \to 1$ as $(\alpha,\beta)\to1$, it holds that 
    \begin{equation*}
        \int_0^{t_{\alpha,\beta}} \sqrt{ \phi'(t) \rho(t,\phi(t)) } \mathrm{d}t 
        \cv{(\alpha,\beta)\to(1,1)} 
        \int_0^1 \sqrt{ \phi'(t) \rho(t,\phi(t)) } \mathrm{d}t \,.
    \end{equation*}
    Moreover, the restriction of $\phi$ to $[0,t_{\alpha,\beta})$ may be seen as an element $\phi_{\alpha,\beta}$ of $\cC^\nearrow(\alpha,\beta)$ by setting 
    \[
        \phi_{\alpha,\beta} := \phi \One{[0,t_{\alpha,\beta})} + \beta^2 \One{[t_{\alpha,\beta},\alpha)} \,,
    \]
    for which it holds that
    \begin{equation*}
        \int_0^{t_{\alpha,\beta}} \sqrt{ \phi'(t) \rho(t,\phi(t)) } \mathrm{d}t 
        = \int_0^{\alpha} \sqrt{ \phi_{\alpha,\beta}'(t) \rho(t,\phi_{\alpha,\beta}(t)) } \mathrm{d}t \,.
    \end{equation*}
    From there, it is straightforward to deduce that
    \begin{align*}
        \liminf_{(\alpha,\beta)\to(1,1)} \sup_{\phi \in \cC^\nearrow(\alpha,\beta)} \int_0^{\alpha} \sqrt{ \phi'(t) \rho(t, \phi(t)) } \mathrm{d}t
        \ge
        \sup_{\phi \in \cC^\nearrow(1,1)} \int_0^{1} \sqrt{ \phi'(t) \rho(t, \phi(t)) } \mathrm{d}t \,,
    \end{align*}
    thus showing \eqref{eq: convergence of DZ constant after rewriting}, and concluding this proof.
\end{proof}

Now, we can properly state and prove our main result.
Some of the main arguments are illustrated in \Cref{fig: LLN zones}.

\begin{theorem}\label{th: LLN with variational constant}
    Let $f$ be a positive, finite, non-increasing, $\cC^1$ function on $[0,1)$.
    Let $\mu_f$ be the associated permuton of \Cref{th: Luce limit permuton} with density $\rho_f$ given by \eqref{eq: density of the Luce permuton}.
    For each $n\in\bN$, let $\bth^{(n)}(f) := \big( f\big(\frac{i-1}{n}\big) \big)_{1\le i\le n}$ be the canonically associated weights.
    Let $\tau_n$ be a $\Luce{\bth^{(n)}(f)}$ random permutation, and $\sigma_{n,\mu_f}$ be a $\mu_f$-sample.
    Assume that
    {$-\log f$ is convex}
    and that $\expec{\LIS{\tau_n}} \asymp \sqrt n$.
    Then:
    \[
        \lim_{n\to\infty} \frac{1}{2\sqrt n} \LIS{\tau_n}
        = \lim_{n\to\infty} \frac{1}{2\sqrt n} \LIS{\sigma_{n,\mu_f}} 
        = \sup_{\phi \nearrow} \int_0^1 \sqrt{ \phi'(t) \rho_f(t,\phi(t)) } \mathrm{d}t
        =: K(f)
    \]
    where the limits hold in probability and in expectation, and the supremum ranges over all non-decreasing, right-continuous curves $\phi:[0,1)\to[0,1)$.
\end{theorem}

\begin{figure}
    \centering
    \includegraphics[width=0.64\linewidth]{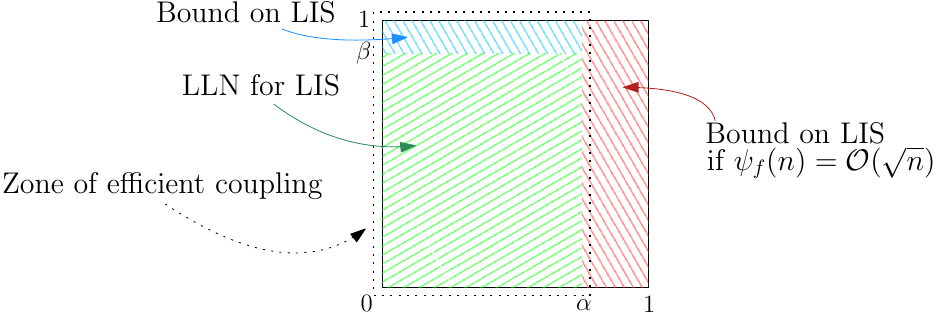}
    \caption{
    Illustration of the main arguments in the proof of \Cref{th: LLN with variational constant}.
    The permuton sample $\sigma_{n,\mu_f}$ is constructed via a $\mu_f$-i.i.d.\ point set in the unit square $[0,1]^2$.
    Green rectangle:
    Proposition~\ref{prop: LLN for restricted permuton samples} identifies the limit constant for the LIS of points in $[0,\alpha]\times [0,\beta]$ for any fixed $\alpha,\beta\in(0,1)$.
    Blue band:
    Corollary~\ref{cor: bound on LIS north-west band} bounds the LIS of points in $[0,\alpha] \times [\beta,1]$ as $\beta = \beta(\alpha) \to 1$.
    Red band:
    in the proof of \Cref{th: LLN with variational constant}, we bound the LIS of points in $[\alpha,1] \times [0,1]$ as $\alpha \to 1$.
    Dotted box:
    Proposition~\ref{prop: coupling discrete and continuous} shows that $\tau_n$ and $\sigma_{n,\mu_f}$ can be efficiently coupled within $[0,\alpha]\times [0,1]$ for any fixed $\alpha\in(0,1)$.
    }
    \label{fig: LLN zones}
\end{figure}

\begin{proof}
    We will assume, w.l.o.g., that $f$ is normalized so that $f(0)=1$.
    Fix $\varepsilon>0$, and write $n_\varepsilon := \lfloor (1-\varepsilon)n \rfloor$.
    Consider the coupling of $\tau_n \sim \Luce{\bth^{(n)}(f)}$ and $\sigma_n \sim \sample_n(\mu_f)$, together with a random subset $\cI_\unmatched \subseteq [n_\varepsilon]$, given by Proposition~\ref{prop: coupling discrete and continuous}.
    We shall argue the following:
    \begin{enumerate}
        \item the LIS of $\tau_n$ and $\sigma_n$ are not concentrated on the positions $i$ close to $n$;
        %\item the LIS of $\tau_n$ and $\sigma_n$ are not concentrated on the values $\tau_n(i)$ and $\sigma_n(i)$ close to $n$;
        \item $\LIS{\tau_n}$ and $\LIS{\sigma_n}$ are asymptotically equivalent;
        \item $\frac{1}{\sqrt n} \LIS{\sigma_n}$ has the claimed limit.
    \end{enumerate}
For $\pi\in\fS_n$ and $1\le a\le b\le n$, we write $\pi^{[a,b]} := \Perm{\pi(i) : a\le i\le b}$.
Note that, with the notation of \Cref{sec: notation}, $\LIS{\pi^{[a,b]}} = \LIS{\left\{ \pi(i) : a\le i\le b \right\}}$.

\medskip\noindent\textbf{1. Bounding the LIS on the last indices.}
By Lemma~\ref{lem: induced Luce}, if $\tau_n \sim \Luce{\bth^{(n)}}$, then $\tau_n^{[n_\varepsilon,n]}$ follows the Luce distribution on $\fS_{n-n_\varepsilon+1}$ with weights $\big( \theta_i^{(n)} \big)_{n_\varepsilon\le i\le n}$.
Therefore, by \Cref{th: LIS growth rate - general} and its proof,
\[
    \expec{\LIS{\tau_n^{[n_\varepsilon,n]}}}
    \le 2e^{3/2} \sum_{k\in \mathbb{Z}}\sqrt{ m_k^{(n,\varepsilon)} }
    \quad\text{where}\quad
    m_k^{(n,\varepsilon)} := \card{ i\in[n_\varepsilon, n] : k \le -\log \theta_i^{(n)} < k+1 }
    \,.
\]
Since the weights are non-increasing, letting $k_\varepsilon := \lfloor -\log \theta_{n_\varepsilon}^{(n)} \rfloor \ge 0$, we get
\begin{equation}\label{eq: LIS last indices with mk}
    \expec{\LIS{\tau_n^{[n_\varepsilon,n]}}}
    \le 2e^{3/2} \sum_{k\ge k_\varepsilon} \sqrt{ m_k^{(n,\varepsilon)} } \,.
\end{equation}
Now, recall that we assume $\theta_i^{(n)} := f\big(\frac{i-1}{n}\big)$.
%and let $(p_k)_{k\ge0}$ be as in Proposition~\ref{prop: growth rate with phi}.
Then, by Proposition~\ref{prop: growth rate with phi} and its proof (namely, \cref{eq: comparison on big p_k,eq: linear approx for small parameters,eq: comparison on small p_k}), we get that
\[
    \sum_{k\ge k_\varepsilon} \sqrt{ m_k^{(n,\varepsilon)} } \le 2 + 2\sum_{k\ge k_\varepsilon}\sqrt{np_k^{(\varepsilon)}} 
    \quad\text{where}\quad
    p_k^{(\varepsilon)} := \Leb\big( A_k \cap [1-\varepsilon,1) \big) \,.
\]
In order to argue that this divided by $\sqrt{n}$ goes to $0$ as $\varepsilon\to0$, we distinguish two cases.
\begin{enumerate}
    \item[(a)] Assume that $\inf_{[0,1)} f > 0$, i.e., $f$ is uniformly bounded below.
    Then $-\log f$ is bounded above on $[0,1)$, and thus, we can choose $\varepsilon > 0$ small enough for $\lfloor -\log f \rfloor$ to be constant on $(1-\varepsilon, 1)$.
    For such choice of $\varepsilon>0$, we have
    \[
        \sum_{k\ge k_\varepsilon} \sqrt{p_k^{(\varepsilon)}}
        = \sqrt{p_{k_\varepsilon}^{(\varepsilon)}}
        \le \sqrt{\varepsilon} 
        \cv{\varepsilon\to0^+} 0 \,.
    \]
    \item[(b)] Assume that $\inf_{[0,1)} f = 0$, i.e., $f(t)\to0$ as $t \uparrow 1$.
    Then, $k_\varepsilon \to \infty$ as $\varepsilon\to0$, and we simply bound ${p_k^{(\varepsilon)}}$ by ${p_k}$ with the notation of Proposition~\ref{prop: growth rate with phi}.
    Thanks to Lemma~\ref{lem: summable condition for sqrt n regime} and our assumption that $\expec{\LIS{\tau_n}} \asymp \sqrt n$, we know that $(\sqrt{p_k})_{k\ge0}$ is summable.
    Therefore,
    \[
        \sum_{k\ge k_\varepsilon} \sqrt{p_k^{(\varepsilon)}}
        \le \sum_{k\ge k_\varepsilon} \sqrt{p_k}
        \cv{\varepsilon\to0^+} 0 \,.
    \]
\end{enumerate}
Coming back to \eqref{eq: LIS last indices with mk}, we conclude that there exists an absolute constant $c_{\varepsilon} \ge 0$ (that only depends on $f$ and $\varepsilon$) such that
\begin{equation}\label{eq: bound on LIS last indices}
    \limsup_{n\to\infty} \frac{1}{\sqrt n} \expec{\LIS{\tau_n^{[n_\varepsilon,n]}}} \le c_{\varepsilon} \,,
\end{equation}
and that $c_{\varepsilon} \to 0$ as $\varepsilon\to0$.
It is clear that the same can be shown for $\sigma_n$, using the same approach complemented by the method behind Lemma~\ref{lem: discrete and continuous models share same groth rate}.

Now, recall that we assumed $\expec{\LIS{\tau_n}} \asymp \sqrt n$, and that the same holds for $\LIS{\sigma_n}$ thanks to Proposition~\ref{prop: growth rate with phi} and Lemma~\ref{lem: discrete and continuous models share same groth rate}.
Hence, we can find an absolute constant $C >0$ such that 
\begin{equation}\label{eq: total sqrt mass}
    \liminf_{n\to\infty} \frac{1}{\sqrt n} \expec{\LIS{\tau_n}}
    \ge C \,,
\end{equation}
and likewise for $\sigma_n$.
In particular, we deduce from \cref{eq: bound on LIS last indices,eq: total sqrt mass} that
\begin{equation}\label{eq: sqrt mass on first indices}
    \liminf_{n\to\infty} \frac{1}{\sqrt n} \expec{\LIS{ \tau_n^{[1,n_\varepsilon]} }}
    \ge C - c_{\varepsilon} 
    = C - o_{\varepsilon}(1) \,,
\end{equation}
and likewise for $\sigma_n$, where $o_{\varepsilon}(1)$ denotes a quantity that goes to $0$ as $\varepsilon\to0$.
From now on, we assume that $\varepsilon>0$ is chosen small enough so that $C-c_{\varepsilon} \ge \frac12 C$.

Finally, since $\tau_n^{[1,n_\varepsilon]}$ is Luce distributed, Corollary~\ref{cor: LIS is close to expectation if non-increasing weights} together with \eqref{eq: sqrt mass on first indices} imply that
\begin{equation}\label{eq: proba sqrt mass on first indices}
    \limsup_{n\to\infty} \prob{
    \LIS{ \tau_n^{[1,n_\varepsilon]} }
    \le \delta \sqrt n
    }
    \cv{\delta \to 0^+} 0 \,,
\end{equation}
i.e., $\LIS{ \left\{ \tau_n(i) : i\in[n_\varepsilon] \right\} }$ is of order $\sqrt n$ in expectation \emph{and} in probability.
Clearly, the same holds for $\sigma_n$.

\medskip\noindent\textbf{2. Comparing the discrete and the continuous models.}
Write 
\begin{align*}
    &\LIS{\left\{ 
    \tau_n(i) : i\in[n_\varepsilon] \setminus \cI_\unmatched
    \right\}}
    \\&\le
    \LIS{\left\{ 
    \tau_n(i) : i\in[n_\varepsilon] 
    \right\}}
    \\&\le
    \LIS{\left\{ 
    \tau_n(i) : i\in[n_\varepsilon] \setminus \cI_\unmatched
    \right\}}
    +
    \LIS{\left\{ 
    \tau_n(i) : i\in \cI_\unmatched
    \right\}} \,.
\end{align*}
Since the middle term is of order $\Theta(\sqrt n)$ by \cref{eq: sqrt mass on first indices,eq: proba sqrt mass on first indices}, item (iii) of our coupling (see Proposition~\ref{prop: coupling discrete and continuous}) implies that
\begin{equation*}
    \frac{\LIS{\left\{ 
    \tau_n(i) : i\in[n_\varepsilon] 
    \right\}}}{\LIS{\left\{ 
    \tau_n(i) : i\in[n_\varepsilon] \setminus \cI_\unmatched
    \right\}}}
    \cv{n\to\infty}
    1 \,,
\end{equation*}
in probability and for the ratio of expectations,
and likewise for $\sigma_n$.
Therefore, using item (i) of Proposition~\ref{prop: coupling discrete and continuous}:
\begin{equation}\label{eq: LIS tau and sigma are equivalent on first indices}
    \frac{\LIS{ \tau_n^{[1,n_\varepsilon]} }}{\LIS{ \sigma_n^{[1,n_\varepsilon]} }}
    \cv{n\to\infty}
    1 \,,
\end{equation}
in probability and for the ratio of expectations.
% Using \cref{eq: bound on LIS last indices,eq: total sqrt mass,eq: sqrt mass on first indices}, we get that
% \begin{equation}
%     \limsup_{n\to\infty}
%     \frac{\expec{\LIS{
%     \tau_n
%     }}}{\expec{\LIS{
%     \sigma_n
%     }}}
%     \le 1 + \frac{c_\varepsilon}{C - c_\varepsilon}
%     \quad\text{and}\quad
%     \liminf_{n\to\infty}
%     \frac{\expec{\LIS{
%     \tau_n
%     }}}{\expec{\LIS{ 
%     \sigma_n
%     }}}
%     \ge \left( 1 + \frac{c_\varepsilon}{C - c_\varepsilon} \right)^{-1} \,.
% \end{equation}
% Moreover, for any given $\delta\in(0,1)$ and $p\in(0,1)$, we can choose $\varepsilon = \varepsilon(\delta,p) >0$ small enough so that
% \begin{equation}
%     \limsup_{n\to\infty} \prob{ \frac{\LIS{\tau_n}}{\LIS{\sigma_n}} \in [1-\delta, 1+\delta] } \ge 1-p \,.
% \end{equation}
% Indeed, assume e.g.\ that $\LIS{\tau_n} \ge (1+\delta) \LIS{\sigma_n}$.

\medskip\noindent\textbf{3. Concluding for the continuous model.}
In this final step of our proof, we shall argue that, for any given $\varepsilon>0$:
\begin{equation}\label{eq: LLN before epsilon}
    2K_\varepsilon^- \le
    \liminf_{n\to\infty} \frac{1}{\sqrt{n}} \LIS{ \sigma_n^{[1,n_\varepsilon]} }
    \le \limsup_{n\to\infty} \frac{1}{\sqrt{n}} \LIS{ \sigma_n^{[1,n_\varepsilon]} }
    \le 2K_\varepsilon^+
\end{equation}
in expectation and in probability, where the constants $K_\varepsilon^-, K_\varepsilon^+$ are such that
\begin{equation}\label{eq: constant asymptotic in epsilon}
    %K_\varepsilon^- \le K 
    %\quad\text{and}\quad
    \lim_{\varepsilon\to0^+} K_\varepsilon^+ 
    = \lim_{\varepsilon\to0^+} K_\varepsilon^- 
    = K \,,
    \quad\text{where}\quad
    K:= \sup_{\phi \nearrow} \int_0^1 \sqrt{ \phi'(t) \rho_f(t,\phi(t)) } \mathrm{d}t \,.
\end{equation}
Before proving \cref{eq: LLN before epsilon,eq: constant asymptotic in epsilon}, let us see how they imply \Cref{th: LLN with variational constant}.
The lower bound is simple to derive:
\begin{equation*}
    \liminf_{n\to\infty} \frac{1}{\sqrt n} \LIS{\sigma_n}
    \ge \sup_{\varepsilon>0} 
    \liminf_{n\to\infty} \frac{1}{\sqrt n} \LIS{\sigma_n^{[1,n_\varepsilon]}}
    \ge 2K \,,
\end{equation*}
in probability and in expectation.
Note that from there, our assumption that $\expec{\LIS{\tau_n}} \asymp \sqrt{n}$ (and thus, the same holds for $\sigma_n$ by Proposition~\ref{prop: growth rate with phi} and Lemma~\ref{lem: discrete and continuous models share same groth rate}) implies that $K<\infty$.
%and likewise for $\tau_n$ thanks to \eqref{eq: LIS tau and sigma are equivalent on first indices}.

The corresponding upper bound is immediate for the expectation, using \eqref{eq: bound on LIS last indices}.
For the probability claim, consider $\delta>0$ and $\varepsilon>0$ small enough and write:
\begin{align*}
    &\prob{ \LIS{\sigma_n} \ge 2(1+3\delta)K\sqrt n }
    \\&\le o(1) + \prob{ \LIS{\sigma_n} \ge (1+2\delta) \LIS{\sigma_n^{[1,n_\varepsilon]}} }
    \\&\le o(1) + \prob{ \LIS{\sigma_n^{[n_\varepsilon,n]}} \ge 2\delta \LIS{\sigma_n^{[1,n_\varepsilon]}} }
    \\&\le o(1) + \prob{ \LIS{\sigma_n^{[n_\varepsilon,n]}} \ge \delta K \sqrt n }
    \\&\le o(1) + \frac{c_\varepsilon}{\delta K} \,,
\end{align*}
where we used \eqref{eq: LLN before epsilon} for the first and third inequalities, \eqref{eq: bound on LIS last indices} for the last inequality, and we let $o(1)$ denote a quantity that goes to $0$ as $n\to\infty$ and may depend on $\varepsilon>0$.
Since this holds for any given $\delta >0$ and $\varepsilon>0$ small enough, and since $c_\varepsilon\to0$ as $\varepsilon\to0^+$, we deduce that 
\begin{equation*}
    \limsup_{n\to\infty} \frac{1}{2\sqrt n} \LIS{\sigma_n}
    \le K 
\end{equation*}
in probability.
The same holds for $\tau_n$ thanks to \eqref{eq: LIS tau and sigma are equivalent on first indices}.

Now, it remains only to prove \cref{eq: LLN before epsilon,eq: constant asymptotic in epsilon}.
For $\alpha,\beta\in(0,1)$, write
\[
\sigma_n^{(\alpha,\beta)} := \Perm{\sigma_n(i) : i\le \alpha n , \sigma_n(i) \le \beta n}
\quad\text{and}\quad
\sigma_n^{\langle \alpha, \beta \rangle} := \Perm{ \sigma_n(i) : i\le \alpha n , \sigma_n(i) \ge \beta n } \,,
\]
and note that, letting $\alpha:= 1-\varepsilon$,
\begin{equation*}
    \LIS{\sigma_n^{(\alpha,\beta)}}
    \le \LIS{\sigma_n^{[1,n_\varepsilon]}}
    \le \LIS{\sigma_n^{(\alpha,\beta)}} + \LIS{\sigma_n^{\langle \alpha,\beta \rangle}} \,.
\end{equation*}
Thus, by Corollary~\ref{cor: bound on LIS north-west band} and Proposition~\ref{prop: LLN for restricted permuton samples},
\begin{equation*}
    2 K_{\alpha,\beta}
    \le \liminf_{n\to\infty} \frac{1}{\sqrt n} \LIS{\sigma_n^{[1,n_\varepsilon]}}
    \le \limsup_{n\to\infty} \frac{1}{\sqrt n} \LIS{\sigma_n^{[1,n_\varepsilon]}}
    \le 2K_{\alpha,\beta} + \eta_{\alpha,\beta} \,,
\end{equation*}
in probability and in expectation, for any given $\alpha := 1-\varepsilon \in (0,1)$ and $\beta\in(0,1)$.
Since $\eta_{\alpha,\beta} \to 0$ as $\beta\to1$, we can choose $\beta := \beta(\alpha)$ such that $\beta\ge\alpha$ and $\eta_{\alpha,\beta} \le 1-\alpha$.
With this choice of $\alpha := 1-\varepsilon$ and $\beta := \beta(\alpha)$, we get \eqref{eq: LLN before epsilon} with $K_\varepsilon^- := K_{\alpha, \beta(\alpha)}$ and $K_\varepsilon^+ := K_{\alpha, \beta(\alpha)} + \frac12 (1-\alpha)$.
Since $K_{\alpha, \beta(\alpha)} \to K$ as $\alpha\to1$ by Proposition~\ref{prop: LLN for restricted permuton samples}, this proves \eqref{eq: constant asymptotic in epsilon} and concludes the proof.
\end{proof}

\subsection{On the variational problem}
\label{sec: variational problem}

\begin{lemma}\label{lem: variational reformulate}
    Let $f:[0,1)\to(0,\infty)$ be a $\cC^1$ 
    %non-increasing
    function.
    Let
    \[
    K(f) := \sup_{\phi:[0,1)\nearrow [0,1)} \int_0^1 \sqrt{\phi'(t) \rho_f(t,\phi(t))} \mathrm dt
    \]
    be the limit constant of \Cref{th: LLN with variational constant}, where the supremum ranges over all non-decreasing, right-continuous curves $\phi : [0,1) \to [0,1)$.
    Then,
    \begin{equation}\label{eq: better variational formula}
    K(f)
    = \sup_{u:[0,1)\nearrow [0,\infty)} \int_0^1 \sqrt{f(t) u'(t) e^{-f(t)u(t)}} \mathrm dt
    \end{equation}
    where the supremum ranges over all non-decreasing, right-continuous curves $u : [0,1) \to [0,\infty)$.
\end{lemma}

\begin{proof}
    Recall from \eqref{eq: function F} and \eqref{eq: density of the Luce permuton} that, for $(x,y) \in (0,1)^2$,
    \[
    \rho_f(x,y) = \frac{ f(x) e^{-f(x) F^{-1}(y)} }{ F'\big( F^{-1}(y) \big) } \,.
    \]
    The claimed formula is then straightforward by considering the mapping
    \[
    {\phi:[0,1)\nearrow [0,1)}
    \qquad\longleftrightarrow\qquad
    {u = F^{-1}\circ \phi :[0,1)\nearrow [0,\infty)} \,.
    \qedhere
    \]
\end{proof}

\begin{remark}
    It would not be difficult to see that we can restrict our attention to the curves ${u:[0,1)\nearrow [0,\infty)}$ that are continuous and such that $u(0)=0$.
    Indeed, each jump in the function would decrease the value of \eqref{eq: better variational formula}.
\end{remark}

Letting $A(u) := \int_0^1 \sqrt{f(t) u'(t) e^{-f(t)u(t)}} \mathrm dt$, we believe that $K(f) = A(u^*_f)$ for a unique $u^*_f \in \cC^1([0,1),[0,\infty))$, although we were not able to prove it and lack a plausible formula for $u^*_f$.
If that could be established, then it would also hold that the longest increasing subsequences in $\tau_n$ and $\sigma_{n,\mu_f}$ (with the notation of \Cref{th: LLN with variational constant}) concentrate around the curve $\phi^*_f := F \circ u^*_f$ in the unit square $[0,1]^2$.

% If $A$ reaches its maximum on $\cC^1_{\nearrow}([0,1])$, then the directional derivatives vanish there.
% If $u\in \cC^1_{\nearrow}([0,1],\bR_+)$ has positive derivative, and $\alpha \in \cC^1([0,1],\bR)$ with $\alpha(0)=0$, then for $\abs{t}$ small enough, $u+t\alpha \in \cC^1_{\nearrow}([0,1],\bR_+)$, and
% \[
%     \left. \frac{\mathrm{d}}{\mathrm{d}t} \right\vert_{t=0} A(u+t\alpha)
%     = \frac12 \int_0^1 \sqrt{f} e^{-\frac12 fu} \left( \frac{\alpha'}{\sqrt{u'}} - f\alpha\sqrt{u'} \right) \,.
% \]
% We would need to solve 
% \[
% \forall \alpha, \quad
% \left. \frac{\mathrm{d}}{\mathrm{d}t} \right\vert_{t=0} A(u+t\alpha) = 0 \,.
% \]
% One naive choice would be $u_0 = \frac{1}{f}-1$ (assuming $f(0)=1$).
% Then,
% \[
% \left. \frac{\mathrm{d}}{\mathrm{d}t} \right\vert_{t=0} A(u_0+t\alpha) = \frac{1}{2} \int_0^1 e^{\frac12(f-1)} \sqrt{\frac{f}{-f'}} (\alpha f)' \,.
% \]
% The function $\beta := \alpha f$ ranges over $\cC^1$ bridges, that is, $\beta(0)=0$, $\beta(1)=0$, $\beta\in \cC^1([0,1],\bR)$.
% Further, $\gamma := \beta'$ ranges over continuous functions with $\int_0^1 \gamma = 0$.
% But then, 
% \[
% \forall \gamma,\quad
% \int_0^1 e^{\frac12(f-1)} \sqrt{\frac{f}{-f'}} \gamma = 0
% \]
% is equivalent to $e^{\frac12(f-1)} \sqrt{\frac{f}{-f'}}$ being constant, which is not the case.
% So, $u_0 = \frac{1}{f}-1$ does not seem to be a local maximum of $A$...

Nonetheless, \eqref{eq: better variational formula} yields non-trivial bounds for $K(f)$.
For example, one might wish to consider the curve $u_1 = F^{-1}$, corresponding to $\phi = \id_{[0,1]}$, for which the computation of 
\[
A(u_1) = \int_0^1 \sqrt{\rho_f(t,t)} \mathrm{d}t
\]
is related to the length of increasing subsequences that are close to the diagonal.
However, computing $A(u_1)$ seems difficult, even for simple functions $f$.
A naive choice that is more tractable is 
%$u_1 := \frac{1}{f}$ and 
$u_0 := \frac{1}{f} - \frac{1}{f(0)}$, for which we find that
\begin{equation}\label{eq: lower bounds on Kf}
    K(f) 
    \ge 
    A(u_0) = \int_0^1 \sqrt{ \frac{-f'(t)}{f(t)} e^{\frac{f(t)}{f(0)}-1} } \mathrm dt
    %> A(u_1) = \int_0^1 \sqrt{ \frac{-f'(t)}{f(t)} e^{-1} } \mathrm dt 
    \,.
\end{equation}
Although $u_0$ does not seem to solve the variational problem of Lemma~\ref{lem: variational reformulate} (and neither does $u_1$; see, also, \Cref{fig: LIS curves}), let us derive the associated lower bounds in two cases:
\begin{itemize}
    \item With the Sukhatme function $f(t) = 1-t$, we find the lower bound
    \[
    K(f) \ge A(u_0) = \frac{2}{\sqrt{e}} \int_0^1 e^{s^2/2} \mathrm ds \,,
    \]
    with corresponding curves $u_0(t) = \frac{t}{1-t}$ and $\phi_0(t) = 1 - \frac{1-t}{t}\big( 1 - e^{-\frac{t}{1-t}} \big)$.
    See \Cref{fig: LIS curves}, middle.
    \item With the Mallows function $f_\beta(t) = e^{-\beta t}$, $\beta\ge0$, we find the lower bound
    \[
    K(f_\beta)
    \ge A(u_0) = \sqrt{\frac{\beta}{e}} \int_0^1 e^{\frac12 e^{-\beta t}} \mathrm dt \,,
    \]
    which can be expressed in terms of the exponential integral function.
    %In particular, $\sqrt{\beta/e} < A(u_0) < \sqrt{\beta}$.
    In particular, we derive the simpler bound $K(f_\beta) > \sqrt{\beta/e}$.
    Recalling that $K(f_\beta)\ge1$ by Corollary~\ref{cor: Luce LIS is at least 2 sqrt n}, we see that these bounds are not optimal in general.\footnote{
    This is unsurprising, considering that $u_0(1) = e^\beta-1 < \infty$ and thus $\phi_0(1) < 1$ in this example.
    More generally, for functions $f$ such that $f(1) > 0$, the fact that $\phi_0(1) < 1$ makes $\phi_0$ a rather bad candidate for the LIS curve.
    }
    Nonetheless, we can compare them with the constant 
    $
    \mathcal{L}_\beta = \beta^{-1/2} \mathrm{arsinh}\big( \sqrt{e^\beta-1} \big)
    $
    from \cite{Mueller_Starr_2013}, and deduce that $K(f_\beta) > \mathcal{L}_\beta$ for large enough $\beta$ (e.g., $\beta\ge7$).
\end{itemize}

\begin{figure}
    \centering
    \includegraphics[width=0.32\linewidth]{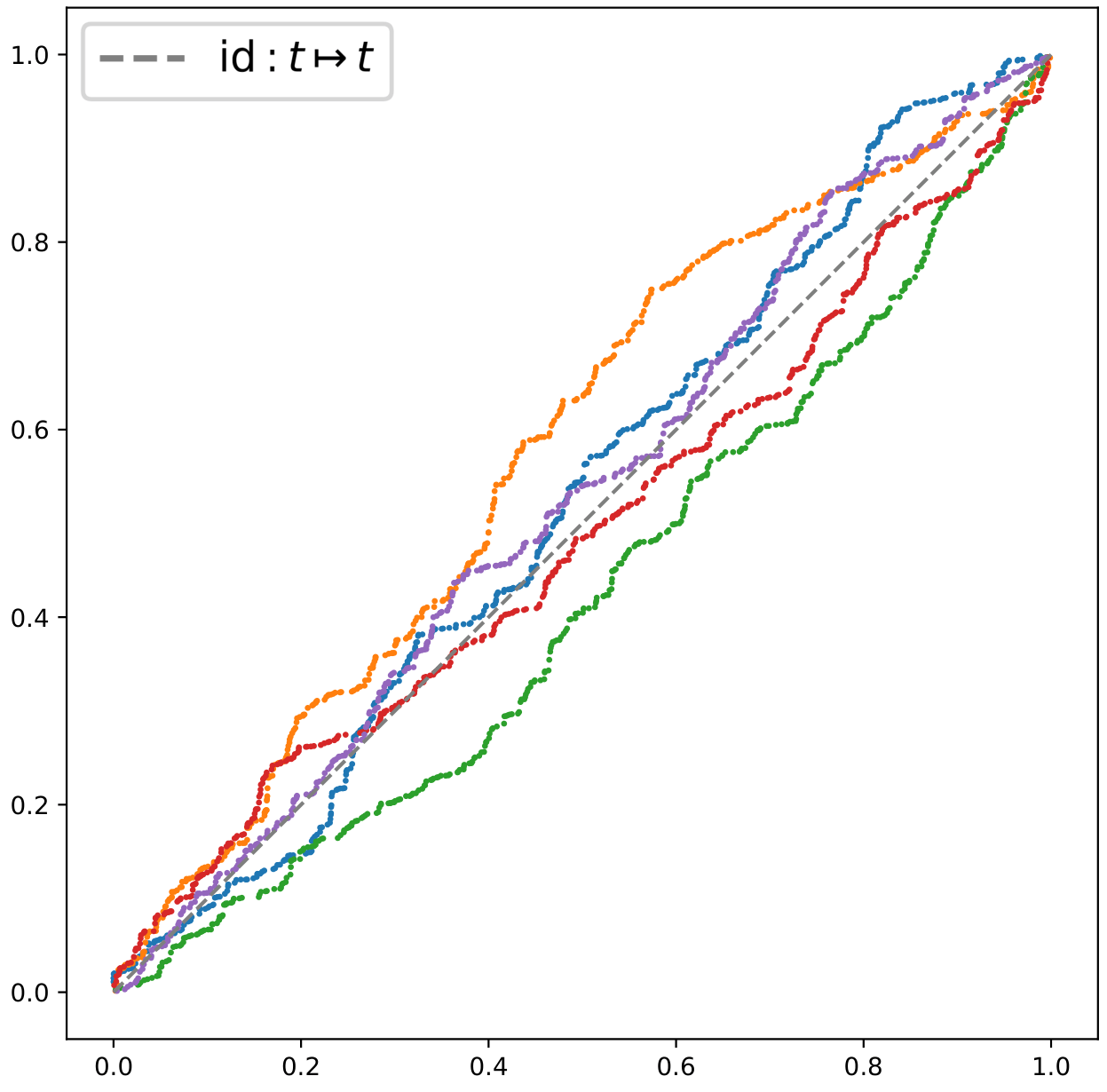}
    \includegraphics[width=0.32\linewidth]{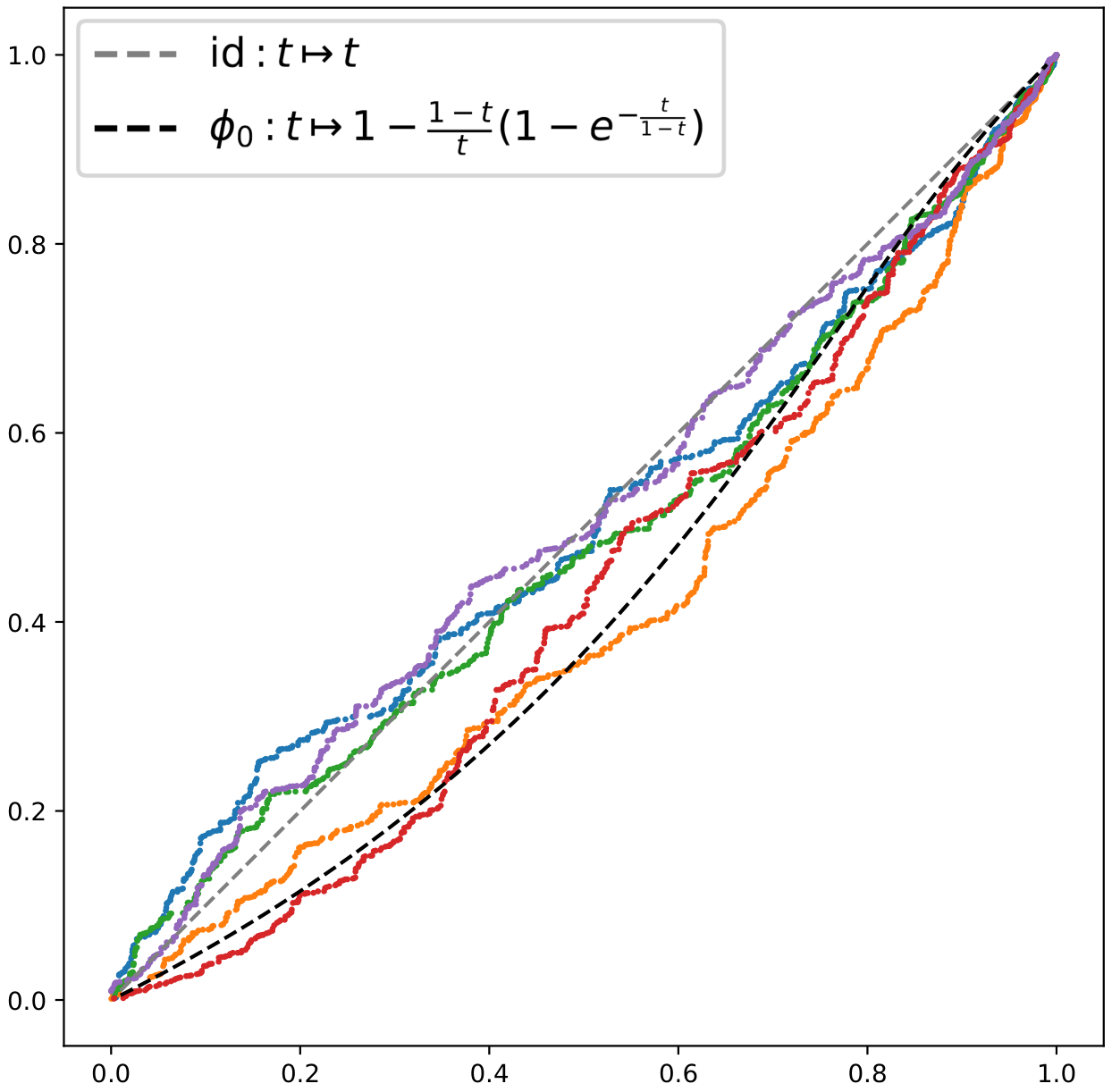}
    \includegraphics[width=0.32\linewidth]{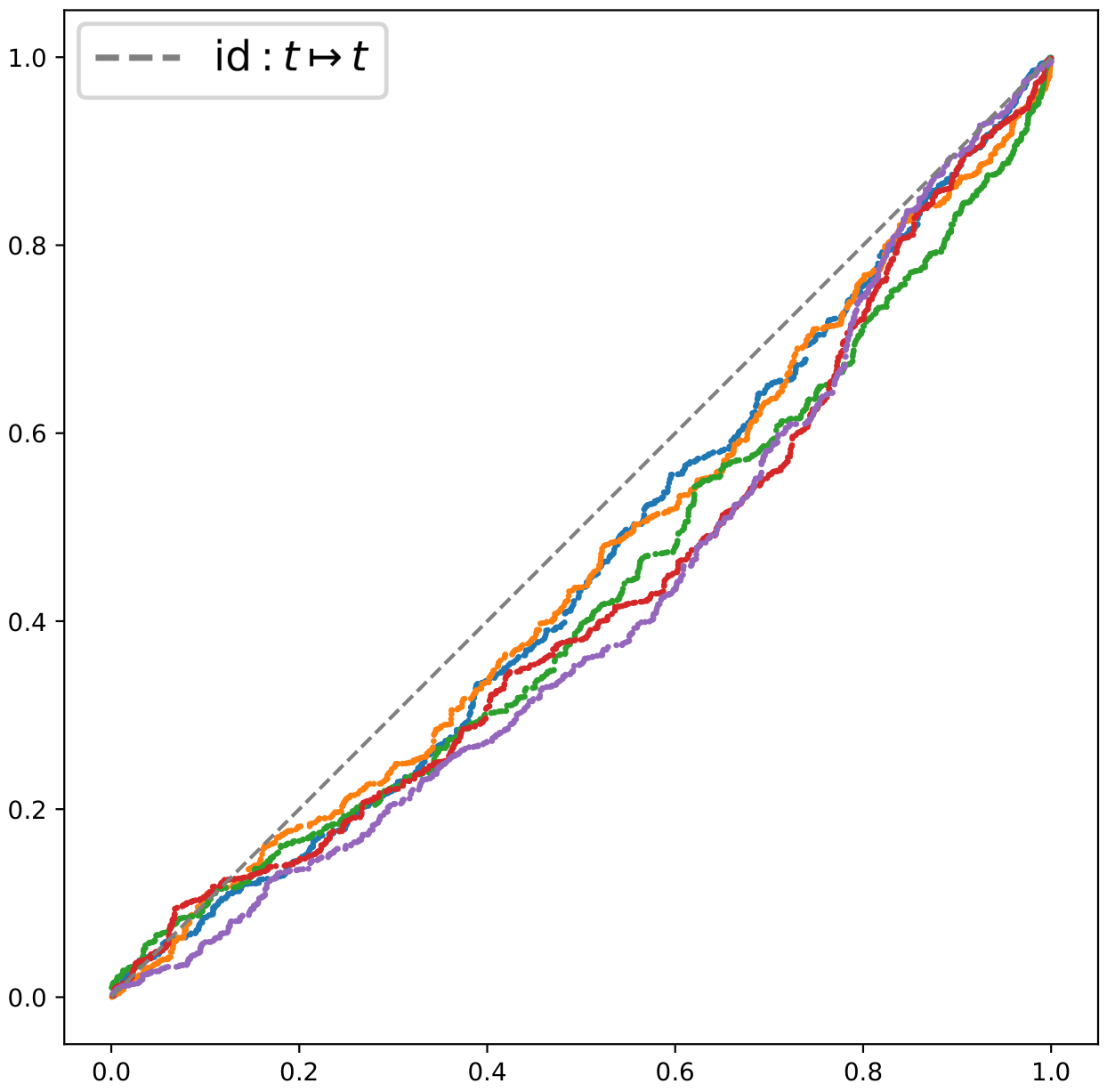}
    \caption{Simulations of the renormalized curves $\left\{ \big( \frac{I_k}{n} , \frac{\tau_n(I_k)}{n} \big) : 1\le k\le \LIS{\tau_n} \right\}$ of longest increasing subsequences in random Luce permutations $\tau_n$.
    The weights are uniform ($f(t)=1$) on the left, Sukhatme ($f(t)=1-t$) in the middle, and Mallows-like ($f_\beta(t) = e^{-\beta t}$) with $\beta=6$ on the right.
    For Sukhatme weights, the (non-optimal) curve $\phi_0 = F \circ u_0$ corresponding to $u_0 = 1/f - 1$ in Lemma~\ref{lem: variational reformulate} is also plotted.
    In all three cases, we sampled $5$ random permutations (each with a different color) of size $52 000$ and plotted an arbitrary longest increasing subsequence (the scattered points).
    In the case of uniform permutations, the longest increasing subsequences are indeed concentrated around the diagonal, and the fluctuations are described by geodesics in the directed landscape \cite{Dauvergne_Virag_2022}.}
    \label{fig: LIS curves}
\end{figure}

\section*{Acknowledgments}

V.D.\ is partially supported by the ANR project LOUCCOUM 
(Large Objects Under Combinatorial Constraints and Outside Uniform Models, 
 ANR-24-CE40-7809). 
J.B.\ is partially supported by NSF grant DMS-2441646 and an Alfred P.\ Sloan Research Fellowship.

\section*{AI Disclosure}

An AI tool was used for literature search and during the revision of this manuscript to check for possible errors, typographical mistakes, and language issues. It was not used to develop the mathematical ideas, results, or proofs.  All suggested changes were reviewed by the authors, who take full responsibility for the final content of the paper.

\appendix

\section{Three probability lemmas}

In this short appendix, we collect three general probability results that we use in our proofs.
The first one is an elementary estimate for the expectation of the square root of a binomial random variable; we include a quick proof for completeness.

\begin{lemma}\label{lem: sqrt of binomial}
    Let $N$ denote a binomial random variable with parameters $n\in\bN$ and {$p\in[0,1]$.}
    Then:
    \[
    \expec{\sqrt{N}} \asymp \min\{ np, \sqrt{np} \}
    \]
    uniformly over $n\in\bN$ and $p \in [0,1]$.
\end{lemma}

\begin{proof}
    For simplicity, we assume that $n\ge4$. 
    Then, we split the proof into two regimes.

    \medskip\noindent\textbf{Small parameters.}
    First, assume that $np\le 3$.
    For the upper bound, observe that $\sqrt{N} \le N$ a.s., and thus, $\expec{\sqrt{N}} \le np$.
    For the lower bound, we have 
    \[
        \expec{\sqrt{N}} \ge np(1-p)^{n-1} \ge np\big( 1 - \tfrac{3}{n} \big)^n \,,
    \]
    and the conclusion easily follows.
    
    \medskip\noindent\textbf{Large parameters.}
    Now, assume that $np > 3$.
    For the upper bound, observe that $\expec{\sqrt{N}} \le \sqrt{np}$ by Jensen's inequality.
    For the lower bound, we have 
    \[
        \prob{ \abs{N - np} \ge \sqrt{2np} }
        \le \frac12
    \]
    by the Bienaymé--Tchebychev inequality, and thus:
    \[
        \expec{\sqrt{N}} 
        \ge \sqrt{ np - \sqrt{2np} } \prob{ N \ge np - \sqrt{2np} }
        \ge \frac12 \sqrt{1-\sqrt{\frac23}} \sqrt{np} \,,
    \]
    and the conclusion follows.
\end{proof}

Then, we state a useful result about the coupling of exponential random variables.

\begin{lemma}\label{lem: exponential coupling}
    Fix $\eta \in (0,1/2)$ and $\delta \in (0,\eta/2)$.
    Let $\theta, \theta' \in [\eta, 1]$, and assume that $\abs{\theta - \theta'} \le \delta$.
    Then, there exists a coupling of $E \sim \Exp{\theta}$ and $T \sim \Exp{\theta'}$ such that the following holds:
    \begin{enumerate}[label=(\roman*)]
        \item $\prob{E \ne T} \le \frac{2\delta}{\eta}\big( 1 + \log\frac{1}{2\delta} \big)$;
        \item the law of $E$ conditionally on the event $\{ E\ne T \}$ has a density $r_{\theta}$ on $[0,\infty)$ satisfying $r_\theta \le R g$, where $R := \frac{1}{\eta} \frac{1-\log({2\delta})}{1-\log({2\delta})+\log\eta}$ and $g$ is a probability density that depends only on $(\delta,\eta)$.
    \end{enumerate}
\end{lemma}

\begin{proof}
    This is a standard approach by (almost) maximal coupling.
    Let $f_1$ and $f_2$ denote the densities of the $\Exp{\theta}$ and $\Exp{\theta'}$ distributions, respectively.
    Let us write $F: (\lambda,x) \in (0,\infty)^2 \mapsto \lambda e^{-\lambda x}$, so that $F(\theta, \cdot) = f_1(\cdot)$ and $F(\theta', \cdot) = f_2(\cdot)$.
    A quick computation shows that $\abs{ \partial_\lambda F(\lambda, x) } \le 2$ for all $(\lambda, x) \in (0,\infty)^2$, and thus:
    \begin{equation*}
        \text{for all } x\ge0,\quad
        \abs{ f_1(x) - f_2(x) } \le 2\delta \,.
    \end{equation*}
    Then, let $h : x\in[0,\infty) \mapsto \max\{ 0 , f_1(x) - 2\delta \}$.
    Let $x_\theta := \frac1\theta \log\frac{\theta}{2\delta} \ge 0$ and $\delta_\theta := \frac{2\delta}{\theta}\big( 1 + \log\frac{\theta}{2\delta} \big)$, so that
    \begin{equation*}
        h = \big( f_1 - 2\delta \big)\One{[0,x_\theta]}
        \quad\text{and}\quad
        \int_0^\infty \big( f_1(x) - h(x) \big) \mathrm dx = \delta_\theta \,.
    \end{equation*}
    We couple $E \sim \Exp{\theta}$ and $T \sim \Exp{\theta'}$ as follows:
    \begin{itemize}
        \item Let $U, P, Q_1, Q_2$ be four independent random variables such that $U$ is uniformly distributed on $[0,1]$, $P$ has density $\frac{1}{1-\delta_\theta} h$, and for $i\in\{1,2\}$, $Q_i$ has density $\frac{1}{\delta_\theta} (f_i - h)$ on $[0,\infty)$;
        \item if $U \le 1 - \delta_\theta$, then we let $E = T = P$;
        \item otherwise, if $U > 1 - \delta_\theta$, we let $E = Q_1$ and $T = Q_2$.
    \end{itemize}
    It is clear that this properly defines a coupling of $E \sim \Exp{\theta}$ and $T \sim \Exp{\theta'}$.
    Furthermore, since $Q_1 \ne Q_2$ a.s.,
    \[
        \prob{E \ne T} = \delta_\theta \le 
        \frac{2\delta}{\eta}\big( 1 + \log\frac{1}{2\delta} \big)
        =: \varepsilon \,,
    \]
    proving the first item.
    Then, conditionally on the event $\{ E \ne T \}$, the random variable $E$ has the density
    \[
        r_\theta := \frac{1}{\delta_\theta}\big( f_1 - h \big)
        = \frac{1}{\delta_\theta}\big( 2\delta\One{[0,x_\theta)} + f_1 \One{[x_\theta,\infty)} \big) \,.
    \]
    Observing that
    \[
        x_\theta \le \frac{1}{\eta}\log\frac{1}{2\delta} =: \xi \,,
    \]
    we can write, for all $x\ge0$,
    \[
        f_1(x) - h(x) \le 
        2\delta\One{[0,\xi)}(x)
        + e^{-\eta x} \One{[\xi, \infty)}(x) =: \tilde g(x) \,,
    \]
    where $\tilde g$ has a total mass $\varepsilon$ on $[0,\infty)$.
    Hence, we can write 
    $
        r_\theta \le \frac{\varepsilon}{\delta_\theta} \cdot \big( \frac{1}{\varepsilon} \tilde g \big)
    $
    and the conclusion follows.
\end{proof}

Finally, we recall the famous Dvoretzky--Kiefer--Wolfowitz (DKW) inequality.

\begin{theorem}\label{th: DKW inequality}
    Let $(X_i)_{i\ge1}$ be a sequence of i.i.d.\ real random variables, with cumulative distribution function $F$.
    For each $n\in\bN$, let
    \[
        F_n : x\in\bR \mapsto \frac1n \sum_{1\le i\le n} \One{X_i \le x}
    \]
    denote the associated empirical distribution function.
    Then, for all $t\ge0$,
    \[
        \prob{ \sup_{x\in\bR} \abs{F_n(x) - F(x)} > t }
        \le 2e^{-2nt^2} \,.
    \]
\end{theorem}

We restate it in the case of uniform random variables for our applications.

\begin{corollary}\label{cor: DKW inequality for uniform variables}
    Let $(U_i)_{1\le i\le n}$ be a sequence of i.i.d.\ uniform random variables on $[0,1]$, and let $U_{(1)} < \ldots < U_{(n)}$ denote their (a.s.) increasing reordering.
    Then, for all $t\ge0$,
    \[
        \prob{ \sup_{1\le i\le n} \abs{U_{(i)} - \frac{i-1}{n}} > t }
        \le 2e^{-2nt^2} \,.
    \]
\end{corollary}

\bibliographystyle{alphaurl}
\bibliography{bibli}

\end{document}